\documentclass[a4paper,10pt]{article}%

\usepackage[utf8]{inputenc}%
\usepackage[T1]{fontenc}%
\usepackage[english]{babel}%

\usepackage{lmodern}%
\usepackage[final]{graphicx}%
\usepackage{xcolor}%
\usepackage{amsmath,amssymb,amsthm,dsfont}
\usepackage{hyperref}%
\usepackage{enumitem}%
\usepackage[labelfont=bf,labelsep=period]{caption}%

\allowdisplaybreaks%
\newcommand{\dr}{\mathrm{d}}%

\newcommand{\Gr}{\mathrm{G}}%
\newcommand{\hgt}{\mathrm{hgt}}%
\newcommand{\Grh}{\Gr_{\hgt}}%
\newcommand{\Grf}{\Grh^{\mathrm{flip}}}%
\newcommand{\Grex}{\Grh^{\mathrm{excl}}}%
\newcommand{\exl}{\mathrm{ex}_{\mathrm{l}}}%
\newcommand{\brl}{\mathrm{br}_{\mathrm{l}}}%

\newcommand{\Nbb}{\mathbb{N}}
\newcommand{\Rbb}{\mathbb{R}}
\newcommand{\Sbb}{\mathbb{S}}
\newcommand{\Zbb}{\mathbb{Z}}

\newcommand{\Sc}{\mathcal{S}}%

\newcommand{\xb}{\mathbf{x}}%
\newcommand{\Xb}{\mathbf{X}}%
\newcommand{\yb}{\mathbf{y}}%
\newcommand{\Yb}{\mathbf{Y}}%
\newcommand{\zb}{\mathbf{z}}%

\newcommand{\opl}{{\color{red}\oplus}}%
\newcommand{\omi}{{\color{blue}\ominus}}%
\newcommand{\bll}{\bullet}%
\newcommand{\wtild}{\widetilde}%
\newcommand{\ordo}{\mathrm{o}}%
\newcommand{\abs}[1]{\left\lvert #1 \right\rvert}%
\newcommand{\lip}[1]{\left\lfloor #1 \right\rfloor}%
\newcommand{\uip}[1]{\left\lceil #1 \right\rceil}%
\newcommand{\ul}[1]{\underline{#1}}%

\newcommand{\oh}{\frac{1}{2}}%
\newcommand{\emp}[1]{\emph{#1}}%

\numberwithin{equation}{section}%

\theoremstyle{plain}%
\newtheorem{theorem}{Theorem}%
\newtheorem{lemma}{Lemma}%
\newtheorem{prop}{Proposition}%
\newtheorem{corollary}{Corollary}%

\theoremstyle{definition}%
\newtheorem{definition}{Definition}%
\newtheorem{remark}{Remark}%

\newcommand{\qqed}{\hfill$\blacksquare$}%

\renewenvironment{proof}[1][Proof]{\begin{trivlist}%
\item[\hskip \labelsep {\bfseries #1}]}{\qqed\end{trivlist}}%

\DeclareMathOperator{\ch}{\mathrm{ch}}%
\DeclareMathOperator{\Exp}{\mathbf{E}}%
\DeclareMathOperator{\Prob}{\mathbf{P}}%
\DeclareMathOperator{\ind}{\mathds{1}}%

\newlist{AAA}{enumerate}{1}%
\setlist[AAA]{label=(A\arabic*),start=0}%

\newcommand{\bristol}{School of Mathematics, University of Bristol, University Walk, Bristol, BS8 1TW, United Kingdom.}%

\newcommand{\prebristol}{Part of this work was done while the first author was affiliated with the MTA-BME Stochastics Research Group and the Alfr\'ed R\'enyi Institute of the Hungarian Academy of Sciences.}%

\newcommand{\bme}{Department of Stochastics, Budapest University of Technology and Economics, Egry J. u. 1., Budapest, H-1111, Hungary.}%

\newcommand{\mta}{MTA-BME Stochastic Research Group (temporary affiliation).}

\newcommand{\titl}{
Dependent Double Branching~Annihilating Random~Walk
}%

\newcommand{\auths}{
M\'arton Bal\'azs\,\footnote{\bristol\ \prebristol}\\[0.1em]
{\normalsize  \url{m.balazs@bristol.ac.uk}}
\and
Attila L\'aszl\'o Nagy\,\footnote{\bme\ \mta}\\[0.1em]
{\normalsize \url{nagyal@math.bme.hu}}
}%

\newcommand{\dat}{\today}%

\title{\titl}%
\author{\auths}%
\date{\dat}%

\begin{document}
\maketitle%

\begin{abstract}
Double (or parity conserving) branching annihilating random walk, introduced in \cite{Sud90}, is a one-dimensional non-attractive particle system in which positive and negative particles perform nearest neighbor hopping, produce two offsprings to neighboring lattice points and annihilate when they meet. Given an odd number of initial particles, positive recurrence as seen from the leftmost particle position was first proved in \cite{BFMP01} and, subsequently in a much more general setup, in \cite{SS08T}. These results assume that jump rates of the various moves do not depend on the configuration of the particles not involved in these moves. The present article deals with the case when the jump rates are affected by the locations of several particles in the system. Motivation for such models comes from non-attractive interacting particle systems with particle conservation. Under suitable assumptions we establish the existence of the process, and prove that the one-particle state is positive recurrent. We achieve this by arguments similar to those appeared in \cite{SS08T}. We also extend our results to some cases of long range jumps, when branching can also occur to non-neighboring sites. We outline and discuss several particular examples of models where our results apply.
\end{abstract}


\newpage

\tableofcontents

\bigskip

\listoffigures

\newpage

\section{Introduction}

\noindent\emp{Models.} Our object of investigation is the process $\Yb$ of \emp{double branching annihilating random walkers} on the integer lattice $\Zbb$. We consider finitely many positive ($\opl$) and negative ($\omi$) particles which are placed on the integer lattice in such a manner that subsequent particles are of opposite type. We only consider the case of an odd total number of particles. It then follows that the charge of the leftmost particle determines the overall (signed) charge of the whole system. The dynamics consists of two parts:
\begin{description}
\item[(RW)]\hfill\\
Particles perform a nearest neighbor random walk on the integer lattice.
\item[(BR)]\hfill\\
Any type of particle can give birth to two offsprings of the same type, placing them to the two neighboring lattice points, while the branching particle changes its type to the opposite.
\end{description}
When two particles of opposite types meet on a site they are simultaneously annihilated. This can happen both for the (RW) and the (BR) steps. Notice that all possible steps conserve
\begin{itemize}
 \item the property that subsequent particles are of opposite type and
 \item the overall charge of the system.
\end{itemize}
One can look at the process $\Yb$ as an \emp{interface} (boundary) process of the \emp{height function} $\Xb$. To see this, define $\Xb$ on the half-integer lattice $\Zbb+\oh$ as the signed, spatially integrated particle number, i.e.\ $\Yb$ is the discrete gradient of $\Xb$. To fix the integration constant, we require that the values of $\Xb$ are either "0" or "1". As an example, $\Xb$ having the Heaviside configuration (all zeros to the left and all ones to the right of the origin) corresponds to a single positive particle of $\Yb$ at the origin. Figure \ref{fig:dbarw-intro} demonstrates another example.
\begin{figure}[!ht]
\centering
\includegraphics[scale=0.3]{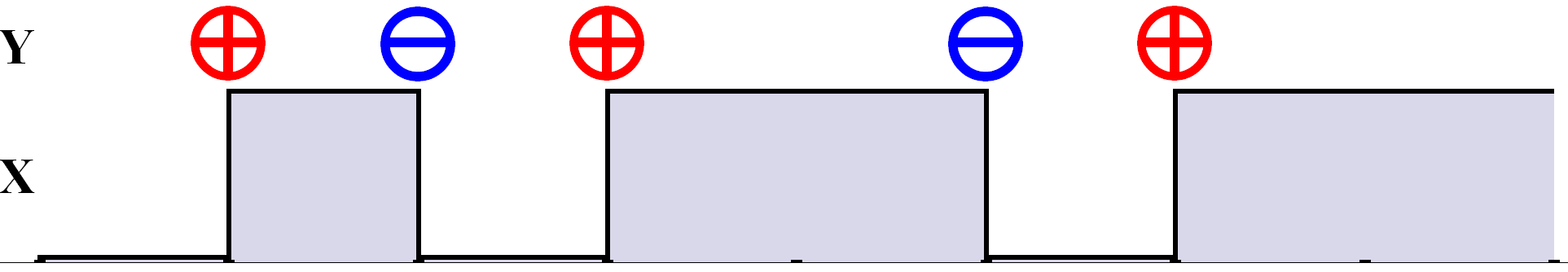}%
\caption[A configuration of double branching annihilating random walkers]
{A possible configuration of double branching annihilating random walkers $\Yb$, and the corresponding height function $\Xb$.}\label{fig:dbarw-intro}%
\end{figure}

It follows that the dynamics described above can be translated into the language of $\Xb$ as:
\begin{description}
\item[(Flip)] (which corresponds to {\bf (RW)})\hfill\\
A height can switch to the value of one of its neighboring heights.
\item[(Excl)] (which corresponds to {\bf (BR)})\hfill\\
An adjacent zero-one (one-zero) pair of heights can exchange values.
\end{description}
When only the \emp{spin-flip} (Flip) jumps are present, and their jump rates are an increasing function of the number of "1" spins close by, the model is the \emp{voter model}. Should the jump rates depend only on nearest neighbor spins, the process becomes the \emp{simple} voter model. A process with only the \emp{exclusion} (Excl) steps is called a \emp{simple exclusion} process. Both of these models were separately under extensive studies in the past decades (see \cite{Lig85,Lig99}). Much less is known when both the (Flip) and the (Excl) types of steps are present in the dynamics at the same time. This is the case one is lead to with the process $\Xb$ that arises from the double branching annihilating walkers $\Yb$. In the sequel we will freely switch between the equivalent descriptions of $\Xb$ and $\Yb$.

Both $\Xb$ and $\Yb$ are in the class of additive and cancellative interacting particle systems (\cite{Gri79}). In many cases a certain monotonicity property is violated in these sort of systems (\cite[pp.\ 71--72, pp.\ 380--384]{Lig85}), making them \emp{non-attractive}. It is due to this lack of monotonicity that direct comparison of the processes via domination arguments is practically impossible. Therefore several problems such as the question of survival or extinction are difficult to treat, and seem to be drastically sensitive to details of the considered branching mechanism (see \cite{BG85,Sud90,BWD91,S00D,S00S,BK11} and further references therein).

\bigskip
\noindent\emp{Earlier results.} Double branching annihilating random walk first appeared in \cite{Sud90}, where it was shown that the process dies out a.s., providing even number of particles are present initially. In terms of $\Xb$, however, one of the first results was a hydrodynamic limit. Considering simple symmetric nearest neighbor exclusion plus translation invariant and finite range dependent rates for the spin-flip jumps, a diffusive scaling leads to the following \emp{reaction--diffusion} type partial differential equation:
\[
\partial_t u(t,x)=\partial_{xx} u(t,x) + F\big(u(t,x)\big),
\]
where $F$ comes from the details of the spin-flip dynamics and $u\in[0,1]$ corresponds to the macroscopic density of particles (\cite{DFL86}, see also \cite[pp.\ 131--132, pp.\ 281--282]{KL99}).

In \cite{BCDFL86} positive recurrence of the Heaviside configuration as seen from the leftmost "1" position was proved for a process with exclusion and asymmetric spin-flip steps, where a "1" neighboring a "0" can result in the "0" flipping into "1" but not vice-versa. This was later generalized to the symmetric spin-flip dynamics ("1" becomes "0" as well) in \cite{BFMP01} for some range of jump rate values. The nearest neighbor condition for the spin-flip dynamics was relaxed in \cite{CD95,BMV07}, but these papers have not dealt with exclusion jumps. \cite{SS08T} then incorporated (long range) exclusion jumps as well in the proof of recurrence. None of the results allowed the jump rates to depend on configuration of particles that are not involved in the jump.

Double branching annihilating random walk is also of main interest in the Physics literature. It is considered as the microscopic model of the following schemes of reactions: $\big\{\opl+\omi\rightarrow\emptyset,\,\opl\rightarrow2\opl+\omi,\,\omi\rightarrow2\omi+\opl\big\}$, which have been under extensive studies in the last decades \cite{CT96,CT98}, and also \cite[Sec.\ 4.6, pp.\ 116--143]{O08}.

\bigskip
\noindent\emp{Motivation and results of this paper}. We first consider the nearest neighbor double branching annihilating random walk process under rather general conditions on how the walking and branching rates depend on the particle configuration. The novelty is that dependence is allowed on states far away from the position where the actual jump takes place.

The motivation for this problem comes from coupling considerations of conservative nearest-neighbor non-attractive interacting particle systems. These are models where particles perform nearest neighbor jumps with rates that depend on the configuration of the initial and the destination site of the jump; no particles are created or annihilated. The most studied versions, like the \emp{asymmetric simple exclusion process}, are attractive: monotonicity of the jump rates exclude the creation of \emp{second class particles} that is, discrepancies between two configurations, in the basic coupling. In slightly more complicated models as e.g.\ the \emp{zero range process} attractivity is not automatic. In non-attractive cases the basic coupling necessarily involves branching of second class particles and second class \emp{anti}particles in a parity-conserving manner. These can later annihilate each other. Moreover, their behavior is heavily influenced by the background process of the first class (that is, ordinary) particles, resulting in an effective attractive potential between the second class particles. It is expected that the second class particles do not proliferate unboundedly even when their number is not conserved in interacting particle systems. This is on the level of intuition, rigorous handling of the second class particle process is currently beyond reach. However, double branching-annihilating random walks with long range dependent rates can serve as a mean-field approximation of the second class particle process, replacing the complicated background process of first class particles by a constant but attractive force between the walkers. Our result is a step towards showing the non-proliferating behaviour, in this mean-field approximation.

In fact, our methods go further than mere attraction. To emphasize the level of generality we mention that examples we can handle include repelling effects in the random walk dynamics as well as unbounded branching rates.

As the construction of the process is far from being trivial under our weak assumptions, we first establish existence of the dynamics by showing that the configuration stays a.s.\ finite for all times.

The main result of the paper is positive recurrence of the singleton state within our general class of long range dependent jump rates for double branching annihilating random walks. It also follows that the stationary distribution sees a finite expected number of particles. We build on methods that were developed in the case of non-dependent rates for the same problem in \cite{SS08T}. These turn out to be robust enough for a significant generalization in terms of jump rate dependence.

Finally we extend the results for long range branching and annihilating processes. In these processes particles perform nearest neighbor, possible dependent random walk, but can put offsprings to non-nearest neighbor places in such a manner that the symmetry of the configuration space is preserved.

\bigskip

\noindent\emp{Organization of the paper.} We give a precise definition of the processes of interest in Section \ref{sec:modnot}. Section \ref{sec:results} involves the main results together with the exact form of our assumptions on the rate functions. Some particular models are outlined in Section \ref{sec:partmodels}. The proofs are postponed to Section \ref{sec:proof}.

\section{The models}\label{sec:modnot}

\subsection{Double branching annihilating random walk}

Define the configuration space
\begin{align}
\Sc=\big\{
\yb\in\{-1,0,+1\}^{\Zbb}\;:\;
 &\textstyle\sum_{i\in\Zbb}\abs{y_i}<\infty \text{ is odd; and if $y_{j}\neq0$, $y_{j'}\neq0$ for some $j<j'$}\nonumber\\
 &\text{such that }\textstyle\sum_{i=j+1}^{j'-1}\abs{y_i}=0, \mbox{ then } y_{j}=-y_{j'}
\big\}.\label{confspaceSc}
\end{align}
We regard $\yb=(y_i)_{i\in\Zbb}\in\Sc$ as a configuration of two types of particles: for an $i\in\Zbb$ we interpret $y_i=+1$ and $y_i=-1$ as the presence of a positive ($\opl$) and negative particle ($\omi$) at lattice point $i$, respectively, while $y_i=0$ means the absence of such particles (hole). The continuous time Markov process
\[
\Yb(t)=(\ldots,Y_{i-1}(t),Y_i(t),Y_{i+1}(t),\ldots)\qquad (t\geq0)
\]
on this space is called the \emp{system of double branching annihilating random walkers} if its (formal) infinitesimal generator can be written in the following way:
\begin{equation}\label{dbarwgenerator}
\Gr := \alpha_1\Gr^{\mathrm{rw}} + \alpha_2\Gr^{\mathrm{br}},
\end{equation}
where $\alpha_1,\alpha_2\geq0$. $\Gr^{\mathrm{rw}}$ acts as
\begin{align}
(\Gr^{\mathrm{rw}} f)(\yb)=
\displaystyle\sum_{i\in\Zbb}
&\left[\ind\{y_i=+1\}
r_i^{\opl}(\yb)+\ind\{y_{i+1}=-1\}\ell_{i+1}^{\omi}(\yb)\right]
(f(\yb-\delta_i+\delta_{i+1})-f(\yb))\nonumber\\
+&\left[\ind\{y_i=-1\}r_i^{\omi}(\yb)+\ind\{y_{i+1}=+1\}
\ell_{i+1}^{\opl}(\yb)\right](f(\yb+\delta_i-\delta_{i+1})-f(\yb)),\label{nnsteppartgenerator}
\end{align}
where $r^{\opl}_{\bll},\ell^{\opl}_{\bll}$ and $r^{\omi}_{\bll},\ell^{\omi}_{\bll}$ correspond to the nearest neighbor, right and left, jumping rates of positive and negative particles, respectively. The branching part $\Gr^{\mathrm{br}}$ of the generator reads as:
\begin{align}
(\Gr^{\mathrm{br}} f)(\yb)=\displaystyle\sum_{i\in\Zbb}
&\ind\{y_i=+1\}b_i^{\opl}(\yb)(f(\yb+\delta_{i-1}-2\delta_i+\delta_{i+1})-f(\yb))\nonumber\\
+&\ind\{y_i=-1\}b_i^{\omi}(\yb)(f(\yb-\delta_{i-1}+2\delta_i-\delta_{i+1})-f(\yb)),
\label{branchingpartgenerator}
\end{align}
where $b^{\opl}_{\bll},b^{\omi}_{\bll}$ are the branching rates of particles having respective charges. Kronecker's delta is denoted by $\delta_i$, that is $\delta_i(j)=1$ if $j=i$ and $0$ otherwise. The generators act on local functions ($f$'s), defined in the usual way (see \cite[Sec.\ 3, Ch.\ I, pp.\ 21]{Lig85}). We emphasize that the nonnegative rates, $r_i^{\circ}(\yb)$'s, $\ell_i^{\circ}(\yb)$'s and $b_i^{\circ}(\yb)$'s ($\circ\in\{\opl,\omi\}$), can depend on the (relative) position of a particle and indeed on the whole configuration of the system, that is in general they can be \emp{non-finite dependent} ones. In Section \ref{sec:results} further regularities will be posed on these rates. Under the right assumptions we will construct the process and show that it a.s.\ has a finite population of particles at any time, which in turn implies that the configuration space \eqref{confspaceSc} is conserved by the dynamics.

Some further notations: let $i_{\mathrm{left}}=\min\{i\in\Zbb:y_i\neq 0\}$ and $i_{\mathrm{right}}=\max\{i\in\Zbb:y_i\neq 0\}$ be the leftmost and rightmost particle position, respectively, of particles. One can assign a \emp{charge} ($\ch$) to a configuration $\yb$ in the following way:
\[
\ch(\yb)=\sum_{i\in\Zbb}y_i.
\]
Notice that in the configuration space $\Sc$ this can only be either $-1$ or $+1$, and agrees with the charge of the leftmost and also of the rightmost particle. Notice that $\ch(\Yb(t))=\ch(\Yb_0)$ holds for every $t\geq0$ by the considered parity conserving branching mechanism.
The \emp{total number of particles} of $\yb\in\Sc$ is denoted by
\[
\abs{\yb}=\sum_{i\in\Zbb}\abs{y_i},
\]
while the particle number process is $(\abs{\Yb(t)})_{t\geq0}$.
Finally we define the \emp{width} $w$ of $\yb$ by letting
\[
w(\yb)=i_{\mathrm{right}}-i_{\mathrm{left}}+1.
\]
In an analogous manner the width process of $(\Yb(t))_{t\geq0}$ is denoted by $W(t)=w(\Yb(t))$.

\subsection{Height function and its interface}\label{sec:heightfunc}

We now set up the \emp{height function} associated with the process $\Yb$. Let $\Sbb$ be the half integer lattice: $\Sbb=\Zbb+\oh$, and for a configuration $\yb\in\Sc$ define $\xb$ by
\[
x_k=\frac{1-\ch(\yb_0)}{2}+\sum_{i=-\infty}^{k-\oh}y_i=\frac{1+\ch(\yb_0)}{2}-\sum_{i=k+\oh}^{+\infty}y_i,
\]
where $k\in\Sbb$. Notice that the sums only have finitely many nonzero terms. This also implies
\begin{equation}\label{heights}
x_k=x_{\oh}+\sum_{i=1}^{k-\oh}y_{i},\qquad x_{-k}=x_{\oh}-\sum_{i=0}^{k-\oh}y_{-i}.
\end{equation}
$\yb\in\Sc$ implies that $\xb$ will be an element of the countable configuration space
\begin{equation*}
\Sc_{\hgt}=\big\{\xb\in\{0,1\}^{\Sbb}:
\text{ the limits }\lim_{i\to-\infty}x_{i+\oh},
\lim_{i\to+\infty}x_{i+\oh}\;\text{ exist and differ}\big\}.
\end{equation*}
In fact, \eqref{heights} provides a one-to-one correspondence between $\Sc$ and $\Sc_{\mathrm{hgt}}$, with inverse relation
\begin{equation}\label{spatderiv}
y_i=x_{i+\oh}-x_{i-\oh}\qquad(i\in\Zbb).
\end{equation}
Notice that $\yb$ defined via \eqref{spatderiv} is in $\Sc$ for every $\xb\in\Sc_{\hgt}$, since an adjacent $01$ ($10$) pair in $\xb$ will result in a $\opl$ ($\omi$) particle in $\yb$. The above discrete gradient expression shows that $\yb$ marks the phase boundaries of the configuration $\xb$. Hence $\yb$ will be called the \emp{interface} associated with $\xb$.

Now define the process $(\Xb(t))_{t\ge0}$ from $(\Yb(t))_{t\ge0}$ as $\xb$ was defined above from $\yb$. It then evolves according to the (formal) infinitesimal generator $\Grh$
\begin{equation}\label{swappingvotergenerator}
\Grh = \alpha_1\Grf+\alpha_2\Grex,
\end{equation}
recalling the parameters $\alpha_1$ and $\alpha_2$ from \eqref{dbarwgenerator}. The spin-flip part $\Grf$ acts on a local function $f$ as
\begin{align}
(\Grf f)(\xb):=
\displaystyle\sum_{k\in\Sbb}
&(1-x_k)\big(x_{k-1}r_{k-\oh}^{\omi}(\yb)+x_{k+1}\ell_{k+\oh}^{\opl}(\yb)\big)\big(f(\xb+\delta_k)-f(\xb)\big)\nonumber\\
+&x_k\big((1-x_{k-1})r_{k-\oh}^{\opl}(\yb)+(1-x_{k+1})\ell_{k+\oh}^{\omi}(\yb)\big)
\big(f(\xb-\delta_k)-f(\xb)\big),\label{voterpartgenerator}
\end{align}
while the exclusion part $\Grex$ is given by
\begin{align}
(\Grex f)(\xb):=\sum_{k\in\Sbb}
&x_{k+1}(1-x_k)b_{k+\oh}^{\opl}(\yb)(f(\xb+\delta_{k}-\delta_{k+1})-f(\xb))\nonumber\\
+&x_k(1-x_{k+1})b_{k+\oh}^{\omi}(\yb)(f(\xb-\delta_{k}+\delta_{k+1})-f(\xb)).\label{exclusionpartgenerator}
\end{align}
In all the above formulae $\yb$ is always attached to $\xb$ in accordance with \eqref{spatderiv}. In order to simplify notation we let for an $l\in\Sbb$:
\begin{equation}
\left.
\begin{aligned}
p_l(\xb)&=
x_l(1-x_{l-1})r_{l-\oh}^{\opl}(\yb)+x_{l-1}(1-x_l)r_{l-\oh}^{\omi}(\yb),\\
q_l(\xb)&=
x_l(1-x_{l-1})\ell_{l-\oh}^{\opl}(\yb)+x_{l-1}(1-x_l)\ell_{l-\oh}^{\omi}(\yb).
\end{aligned}\label{ratepq}
\right\}
\end{equation}
With these notations it is easy to see that the rate at which the $k^{\mathrm{th}}$ bit in $\Xb$ changes to the opposite is $p_{k}(\Xb)+q_{k+1}(\Xb)$, while with the subsequent bit they exchange their values (swap) with rate $b_{k+\oh}^{\opl}(\Yb)$ or $b_{k+\oh}^{\omi}(\Yb)$ according to whether $0$ and $1$ or $1$ and $0$ the $k^{\mathrm{th}}$ and $(k+1)^{\mathrm{th}}$ bits in $\Xb$ were, respectively. The evolution of $\Yb$ and $\Xb$ is indicated in Figure \ref{fig:dbarw-model}.
\begin{figure}[!ht]
\centering
\includegraphics[scale=0.3]{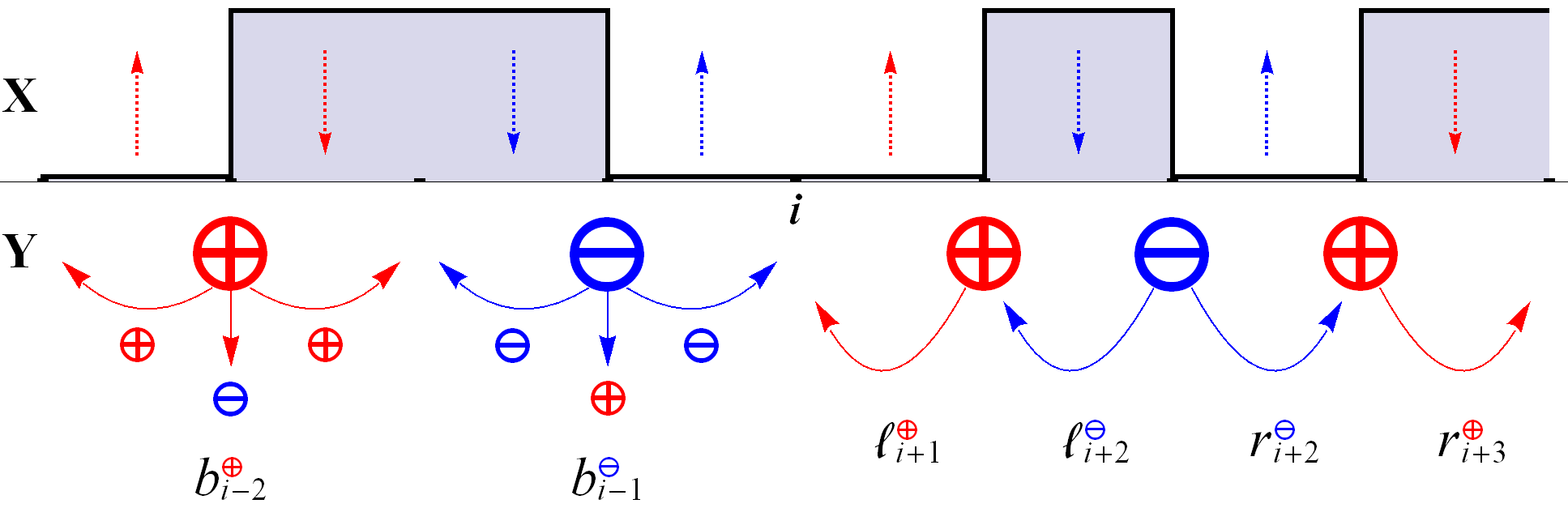}%
\caption[Transitions of double branching annihilating random walkers with heights]
{The possible transitions and the corresponding rates of double branching annihilating random walkers $\Yb$ with the height function $\Xb$.}\label{fig:dbarw-model}%
\end{figure}
Notice also that by our definitions the height function $\Xb$ of $\Yb$ can be regarded as the \emp{(signed) space-time integrated particle flux} of $\Yb$: the value of $-X_k(t)$ increases by one for each positive, and decreases by one for each negative particle that crosses from left to right the space-time path connecting $(\oh,0)$ and $(k,t)$. Particles crossing from right to left have the opposite effect on $-X_k(t)$. In particular, when $\ch(\Yb_0)=-1$ ($\ch(\Yb_0)=+1$) and there are no particles to the right (left) of the origin initially, then $X_{\oh}(0)=0$ and $-X_{\oh}(t)$ counts the total signed charge that jumped over the bond $(0,1)$ up to time $t$.

Finally we define
\[
\wtild{\Yb}=(\wtild{\Yb}(t))_{t\geq0}=(\wtild{Y}_0(t),\wtild{Y}_1(t),\ldots,\wtild{Y}_n(t),\ldots)_{t\geq0}
\]
as $(\Yb(t))_{t\geq0}$ as seen from its leftmost particle position, that is $\wtild{Y}_n(t)=Y_{i_{\mathrm{left}}+n}(t)$ ($t\geq0$), for every $n\in\Zbb^+_0$. Hence the configuration space of $\wtild{\Yb}$ is
\begin{equation*}
\wtild{\Sc}:=\big\{\wtild{\yb}\in\Sc \text{ such that $\wtild{y}_0\neq0$ and $\wtild{y}_i=0$ if $i<0$}\big\}.
\end{equation*}
Similarly for the height function,
\[
\wtild{\Xb}=(\wtild{\Xb}(t))_{t\geq0}=(\wtild{X}_{\oh}(t),\wtild{X}_{\oh+1}(t),\ldots,\wtild{X}_{m}(t),\ldots)_{t\geq0}
\]
is the process of $(\Xb(t))_{t\geq0}$ as seen from its leftmost discrepancy, that is $\wtild{X}_{m}(t)=X_{i_{\mathrm{left}}+m}(t)$ ($t\geq0$), where $m\in\Sbb^+$ is a positive half integer. The corresponding configuration space reads as:
\begin{align*}
\wtild{\Sc}_{\hgt}=\big\{\wtild{\xb}\in\Sc_{\hgt}\text{ such that }
\text{either }&\wtild{x}_{-\oh}=0=1-\wtild{x}_{\oh}\text{ and $\wtild{x}_{i-\oh}=0$ if $i<0$}\\
\text{or }&\wtild{x}_{-\oh}=1=1-\wtild{x}_{\oh}\text{ and $\wtild{x}_{i-\oh}=1$ if $i<0$}\big\}.
\end{align*}
It is these processes for which recurrence statements can be established.
Recall that a countable Markov process is said to be \emp{ergodic} iff it has a unique stationary distribution.
In Section \ref{sec:results} assumptions \eqref{assump:R21a} (from \ref{A1}), \eqref{assump:R22} (from \ref{A2}) and \ref{A0} will guarantee $\wtild{\Xb}$ ($\wtild{\Yb}$) to be an irreducible, countable Markov process on $\wtild{\Sc}$ ($\wtild{\Sc}_{\hgt}$), respectively. Hence positive recurrence will be equivalent to ergodicity in our context. When $\wtild{\Xb}$ ($\wtild{\Yb}$) is positive recurrent then it is also said to be \emp{interface tight} (\emp{stable}), cf.\ \cite[Def.\ 7, pp.\ 63]{SS08V} (\cite[pp.\ 166]{SS08T}), respectively.

\section{Main results}\label{sec:results}

All results concerning the process $(\Yb(t))_{t\geq0}$ are collected in this section. Further results on an extension of the model are also provided in Subsection \ref{sec:nnnbranching}. From now on we denote by $\Yb_0=\Yb(0)\in\Sc$ and $\Xb_0=\Xb(0)\in\Sc_{\hgt}$ a fixed initial configuration of the process $(\Yb(t))_{t\geq0}$ and $(\Xb(t))_{t\geq0}$, respectively.

In this section we will take a total of six assumptions (\ref{A0}--\ref{A5}), concerning the rate functions which appeared in the definition of infinitesimal generators \eqref{nnsteppartgenerator} and \eqref{branchingpartgenerator}. We will make comments and explain these assumptions in detail, however, the exact role of them will only become clear in the details of proofs in Section \ref{sec:proof}. We formulate the first three assumptions which are needed for the first two subsections.
\begin{AAA}
\item\label{A0} \emp{Translation invariance}.\\
For every $i\in\Zbb$ we have
\[
r_{i}^{\circ}(\yb)=r_{i+1}^{\circ}(\zb),\, l_{i}^{\circ}(\yb)=l_{i+1}^{\circ}(\zb),\text{ and }\,
b_{i}^{\circ}(\yb)=b_{i+1}^{\circ}(\zb)\qquad(\circ\in\{\opl,\omi\}),
\]
whenever $\yb,\zb\in\Sc$ are such that $z_j=y_{j-1}$ ($j\in\Zbb$).
\item\label{A1} \emp{Bounds on the random walk rates}.\\
We require
\begin{subequations}
\begin{align}
\ul{s}&=\inf_{i\in\Zbb,\,\yb\in\Sc:y_i\neq0}\big[\min\{r_{i}^{\opl}(\yb),\ell_{i}^{\opl}(\yb)\}+\min\{r_{i}^{\omi}(\yb),\ell_{i}^{\omi}(\yb)\}\big] \label{assump:R21a}\\
&<\sup_{i\in\Zbb,\,\yb\in\Sc:y_i\neq0}\big[r_{i}^{\opl}(\yb)+\ell_{i}^{\opl}(\yb)+r_{i}^{\omi}(\yb)+\ell_{i}^{\omi}(\yb)\big]\leq 1\label{assump:R21b}
\end{align}
\end{subequations}
to hold for some $0< \ul{s} < 1$.
\item\label{A2} \emp{Bounds and asymptotics on the branching rates}.\\
First
\begin{equation}\label{assump:bdrift}
\ch(\yb)\left[\sum_{i\in\Zbb}b_{i}^{\opl}(\yb)-\sum_{j\in\Zbb}b_{j}^{\omi}(\yb)\right] \leq \bar{d}\abs{\yb}
\end{equation}
holds for some $0\leq\bar{d}\in\Rbb$. Second, for every $n\in\Zbb^+$
\begin{subequations}\label{assump:R22}
\begin{align}
0<&\inf_{i\in\Zbb,\,\yb\in\Sc:y_i\neq0,\,\abs{\yb}\leq n}\big[b_{i}^{\opl}(\yb)+b_{i}^{\omi}(\yb)\big],
\label{assump:R22a}\\
&\sup_{i\in\Zbb,\,\yb\in\Sc:\abs{\yb}=n}\big[b_{i}^{\opl}(\yb)+b_{i}^{\omi}(\yb)\big]
\leq B_n<+\infty\label{assump:R22b}
\end{align}
\end{subequations}
holds. Moreover, let $B(N):=\max_{1\leq n\leq N}(n\cdot B_n)$ for $N\in\Zbb^+$, then
\begin{subequations}\label{assump:rateB}
\begin{align}
\sum_{N=1}^{+\infty}\frac{1}{B(N)}=+\infty,\text{ and there exists}&\text{ a $0<\bar{D}$ such that}\label{assump:rateBa}\\
\limsup_{N\to+\infty}\frac{\max_{1\leq n\leq N}(B_n)}{N}&<\bar{D}<+\infty.\label{assump:rateBb}
\end{align}
\end{subequations}
\end{AAA}
We make some comments on the above assertions.
\begin{remark}\leavevmode
\begin{itemize}
\item Assumptions \eqref{assump:R21a} and \eqref{assump:R22a} express, in some sense, the non-vanishing of the dynamics for any configuration.
\item The constant $1$ for the upper bound in \eqref{assump:R21b} is of course artificial, and can be replaced with any positive number by rescaling time.
\item The decay condition posed to the function $B(\bll)$ in \eqref{assump:rateBa} is essential for the process to be well-defined (see Proposition \ref{propwelldefined}), while \eqref{assump:rateBb} is necessary for the width estimate (see Proposition \ref{propvariance}). This latter one together with \eqref{assump:bdrift} and the bounds in \eqref{assump:R21b} will be used in the proof of Lemma \ref{lemma} as well, thus constitute as cornerstones to the recurrence results in Subsection \ref{sec:tightness}.
\item Notice that in \ref{A2}, assumptions \eqref{assump:rateBa} and \eqref{assump:rateBb} are (logically) independent of each other in the sense that neither of them implies the other.
\end{itemize}
\end{remark}

\subsection{Bounds on the number of steps and the width process}\label{sec:construction}

In the first assertion under mild conditions we claim that our process is well-defined.
\begin{prop}\label{propwelldefined}
Assume \eqref{assump:R21b} from \ref{A1}, and that \eqref{assump:R22b}, \eqref{assump:rateBa} hold from \ref{A2}. Then the process $(\Yb(t))_{t\geq0}$ takes a.s.\ only finitely many steps in every finite time interval. In particular it follows that $\Yb(t)\in\Sc$ holds for every $t>0$ as well.
\end{prop}
Notice that \eqref{assump:rateBa} itself does not imply any bound on the growth of the branching rates. For instance let $\iota(x):=x(1+2^x)$, where $x\in\mathbb{N}$. Then the function
\begin{equation*}
B(N)=\sum_{j=0}^{+\infty}\ind\big\{\iota^{(j)}(2)\leq N<\iota^{(j+1)}(2)\big\}
\,\iota^{(j)}(2)\cdot 2^{\iota^{(j)}(2)}\quad (N\in\mathbb{N})
\end{equation*}
satisfies \eqref{assump:rateBa}, where $\iota^{(j)}$ is the $j^{\mathrm{th}}$ iterate of $\iota$ ($j\in\Zbb^+$), and $\iota^{(0)}\equiv1$. Indeed, $B$ takes the constant value $\iota^{(j)}(2)\cdot 2^{\iota^{(j)}(2)}$ on the interval $[\iota^{(j)}(2),\,\iota\big(\iota^{(j)}(2)\big)]$, hence the sum of $1/B(N)$ values above each of these intervals is $1$. Now, setting the branching rates $b^{\opl}_i(\yb),\,b^{\omi}_i(\yb)$ be equal to $B(n)/n$ whenever $\abs{\yb}=n\geq 1$ and $y_i\neq 0$, it easily follows that along the subsequence $N_j=\iota^{(j)}(2)$,
$B\big(\iota^{(j)}(2)\big)/(\iota^{(j)}(2)\big)^2$ blows up as $j\to+\infty$, hence condition \eqref{assump:rateBb} cannot hold. Of course, the exponential choice in the definition of function $\iota$ was artificial and can be replaced by any other fast increasing sequence.

The above example thus shows that \eqref{assump:rateBa} might not prevent the expected size of the width of $(\Yb(t))_{t\geq0}$ from blowing up in finite time. However, the situation is much better under \eqref{assump:rateBb}, namely, one can establish an estimate to the width process $(W(t))_{t\geq0}$.
\begin{prop}\label{propvariance}
Under assumptions \eqref{assump:R21b} from \ref{A1}, \eqref{assump:R22b} and \eqref{assump:rateBb} from \ref{A2}, the $r^{\mathrm{th}}$ moment of both $\abs{\Yb(t)}$ and $W(t)$ are finite for each $t>0$ and for every $r>0$.
\end{prop}
As a consequence of these assertions we can conclude that under assumptions \eqref{assump:R21b}, \eqref{assump:R22b} and \eqref{assump:rateB}, the number of steps taken by $(\Yb(t))_{t\geq0}$ is a.s.\ finite as well as its width process has finite expectation for every time.

We notice that the above results translate naturally to ones for the process $(\Xb(t))_{t\geq0}$, as well.

\subsection{Positive recurrence of the singleton state}\label{sec:tightness}

We begin with introducing two other conditions, with explanations following them. First
\begin{AAA}[resume]
\item\label{A3} \emp{Impact and long range dependence of the interaction}.\\
For every $L\in\Zbb^+$, $\hat{\yb}\in\Sc$ and $i\in[i_{\mathrm{left}},i_{\mathrm{right}}]$, we have
\[
|r_i^{\circ}(\hat{\yb})-r_{i}^{\circ}(\yb)|+
|\ell_i^{\circ}(\hat{\yb})-\ell_{i}^{\circ}(\yb)|+
|b_i^{\circ}(\hat{\yb})-b_{i}^{\circ}(\yb)| \leq H(\abs{\yb},L),
\]
where $i_{\mathrm{left}}$ and $i_{\mathrm{right}}$ are the leftmost and rightmost particle positions of $\hat{\yb}$, respectively, $\circ\in\{\opl,\omi\}$, and $\yb\in\Sc$ is such that
\begin{description}
\item if $\ch(\hat{\yb})=+1$, then $(\hat{y}_i-y_i)\ind\{i\in[i_{\mathrm{left}}-L,+\infty)\}=0$, while
\item if $\ch(\hat{\yb})=-1$, then $(\hat{y}_i-y_i)\ind\{i\in(-\infty,i_{\mathrm{right}}+L]\}=0$.
\end{description}
Furthermore, $H:\Zbb^+\times\Zbb^+\to\Rbb^+_0$ from above satisfies
\begin{equation}\label{assump:3decay}
\lim_{L\to+\infty}H(N,L)=0
\end{equation}
for every fixed $N\in\Zbb^+$.
\begin{figure}[!ht]
\centering
\includegraphics[scale=0.16]{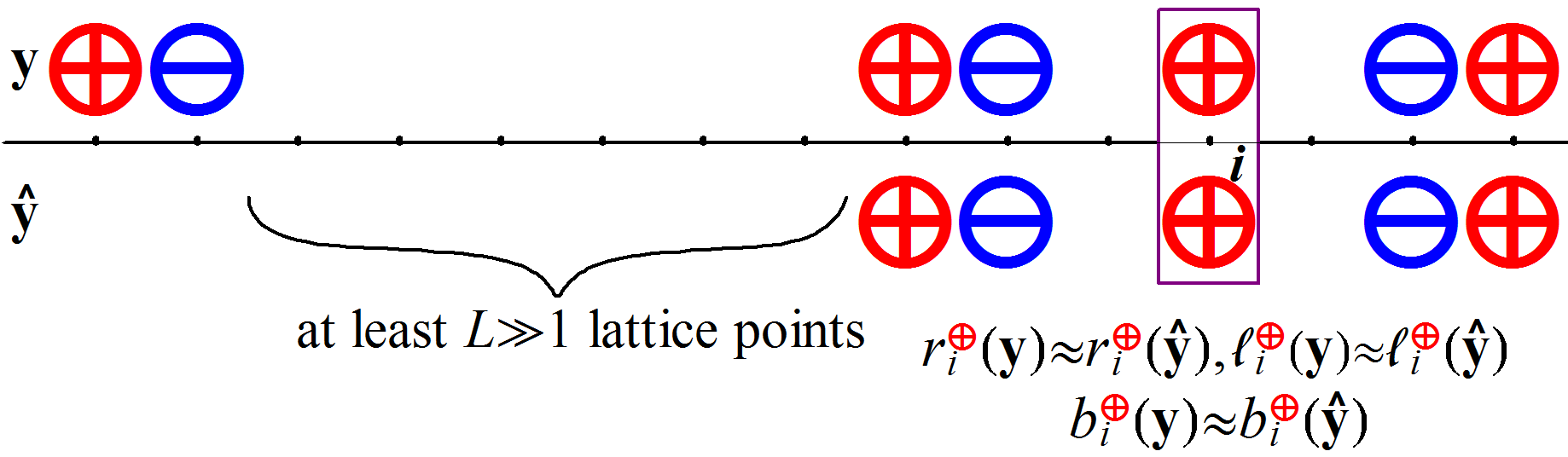}%
\caption[The meaning of assumption \ref{A3}]
{Condition \ref{A3} expresses that if a bunch of particles form an isolated island between particles far away at least one direction then the ``inner'' system of particles is insensitive of remote changes.}\label{fig:dbarw-assump-A3}%
\end{figure}

\item\label{A4} \emp{Strength of the interaction between random walkers}.
\begin{enumerate}[label=(A4\alph*)]
\item\label{A4a}
For every fixed $\xb\in\Sc_{\hgt}$ and $k<l\in\Sbb$ we have
\begin{equation}\label{assump:attraction}
p_k(\xb)+q_l(\xb) \geq q_k(\xb)+p_l(\xb),
\end{equation}
where there are no particles in $(k,l)$ and none of the rates ($p$, $q$) vanish either on $k$ or on $l$, \\
\emp{or}
\item\label{A4b}
\eqref{assump:attraction} holds only for $k<l\in\Sbb$ for which $y_{k-\oh}=\opl$ and $y_{l-\oh}=\omi$.

In addition we also have:
\begin{equation}\label{eq:supplA4}
\begin{array}{cll}
&p_{i_{\mathrm{right}}+\oh}(\xb)=\displaystyle\min_{k\in\Sbb,\,p_k(\xb)\neq0}p_k(\xb),
&\;q_{i_{\mathrm{right}}+\oh}(\xb)=\displaystyle\max_{k\in\Sbb,\,q_k(\xb)\neq0}q_k(\xb)\;\text{ and }\\
&p_{i_{\mathrm{left}}+\oh}(\xb)=\displaystyle\max_{k\in\Sbb,\,p_k(\xb)\neq0}p_k(\xb),
&\;q_{i_{\mathrm{left}}+\oh}(\xb)=\displaystyle\min_{k\in\Sbb,\,q_k(\xb)\neq0}q_k(\xb),
\end{array}
\end{equation}
if $\ch(\yb)=+1$ and $\ch(\yb)=-1$, respectively.
\begin{figure}[!ht]
\centering
\includegraphics[scale=0.16]{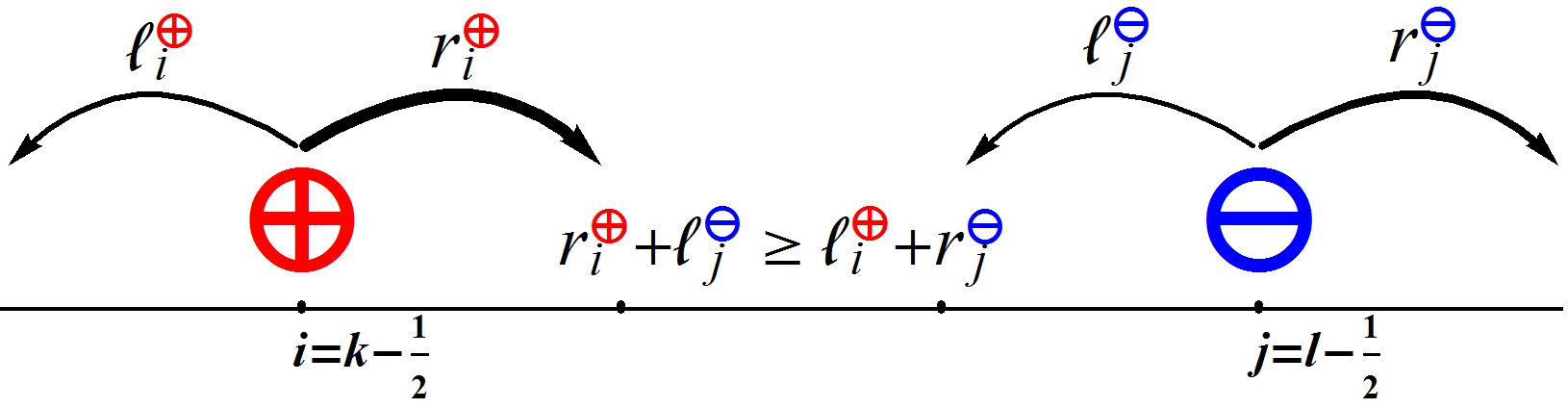}%
\caption[An illustration to assumption \ref{A4}]
{Condition \ref{A4} states that two consecutive particles are more likely to jump towards, rather than against each other.}\label{fig:dbarw-assump-A4}%
\end{figure}
\end{enumerate}
\end{AAA}

A bit more explanations are required to be appended to \ref{A3} and \ref{A4} since these assumptions are far from being technical ones.
In words \ref{A3} tells us that for a fixed bunch of particles situated close to each other, having width $i_{\mathrm{right}}-i_{\mathrm{left}}+1$, the effect of particles far away is decaying in distance. For an illustration of condition \ref{A3} see Figure \ref{fig:dbarw-assump-A3}. For positive recurrence it is clear that there cannot be odd number of straggling particles far from the others: eventually the leftmost particle must find the rightmost one. The ``attraction'' is incorporated to the interaction by \ref{A4a} (see Figure \ref{fig:dbarw-assump-A4}). However, this can be weakened and it is still possible to work with repulsive effects between particles under \ref{A4} (see \ref{A4b} and the related models of Section \ref{sec:partmodels}).

The result of this section is the following theorem. In there we have to impose one more condition that connects the general strength of the random walk part of the dynamics and that of the branching dynamics. This is needed in our proofs to make sure particles do meet often enough. Notice that this condition still allows unbounded branching rates in some specific manner (see \ref{A2} and the rate instances in the third paragraph of Subsection \eqref{sec:brrates}). It would be interesting to see if, and how, ergodicity fails when this condition is not met, but this question is beyond the reach of our current techniques.
\begin{theorem}\label{tightness}
Under assumptions \ref{A0}--\ref{A4} and further assuming that $\alpha_1\ul{s}>2\alpha_2\bar{d}$ holds, the processes $(\wtild{\Yb}(t))_{t\geq0}$ and $(\wtild{\Xb}(t))_{t\geq0}$ (see Subsection \ref{sec:heightfunc}) are ergodic on $\wtild{\Sc}$ and on $\wtild{\Sc}_{\hgt}$, respectively.
\end{theorem}
\noindent Some interesting models that fulfill all the conditions of Theorem~\ref{tightness} will be given in Section \ref{sec:partmodels}. Now, as a consequence we also have the following:
\begin{corollary}\label{corollary:tightness}\leavevmode
\begin{enumerate}
\item There exists a unique, positive probability distribution $(\pi(\wtild{\yb}))_{\wtild{\yb}\in\wtild{\Sc}}$ on $\wtild{\Sc}$, which is the only invariant measure of the process $(\wtild{\Yb}(t))_{t\geq0}$ and
\[
\lim_{t\to+\infty}\Prob(W(t)=n)=
\textstyle\sum_{\wtild{\yb}\in\wtild{\Sc}\,:\,w(\wtild{\yb})=n}\!\pi(\wtild{\yb}),
\]
implying that the width process is stochastically compact (tight), where $n\in\Zbb^+$.
\item At equilibrium the particle number is of finite expectation, i.e.:
\[
\textstyle\sum_{\wtild{\yb}\in\wtild{\Sc}}\abs{\wtild{\yb}}\pi(\wtild{\yb})<+\infty.
\]
\item Finally, the width process grows at most linearly, that is:
\begin{equation}\label{widthestimate}
\limsup_{t\to+\infty}\frac{\Exp[W(t)]}{t}<+\infty.
\end{equation}
\end{enumerate}
\end{corollary}
Finally we notice that we have not investigated models in which the condition $\alpha_1\ul{s}>2\alpha_2\bar{d}$ is violated including the case when $\alpha_1=0$ but $\bar{d}\geq0$. Also we have not considered non-nearest neighbor walking particles but we extended the above result to processes with long range branching mechanisms, see the details in the forthcoming Subsection \ref{sec:nnnbranching}.

\subsection{Extensions}\label{sec:nnnbranching}

In this subsection we extend the results of the previous sections to double branching annihilating random walkers with non-nearest neighbor branching mechanisms. We say that the process $\left(\Yb^{(\mathrm{LR})}(t)\right)_{t\geq0}$ on $\Sc$ is a \emp{parity conserving long range branching and annihilating random walk} if its (formal) infinitesimal generator can be written in the following way:
\[
\Gr^{(\mathrm{LR})}=\Gr+\sum_{l=2}^{+\infty}\Gr^{\brl},
\]
$\Gr$ was defined in \eqref{dbarwgenerator}. $\Gr^{\brl}$ acts on a local function $f$ as
\[
\begin{aligned}
(\Gr^{\brl} f)(\yb)=\sum_{i\in\Zbb}
&\ind\{y_i=+1\}b_{i,l}^{\opl}(\yb)\bigg(\prod_{\substack{j=i-l+1\\j\neq i}}^{i+l-1}\!\!(1-\abs{y_{j}})\bigg)
\big(f\left(\yb+\delta_{i-l}-2\delta_i+\delta_{i+l}\right)-f(\yb)\big)\\
+&\ind\{y_i=-1\}b_{i,l}^{\omi}(\yb)\bigg(\prod_{\substack{j=i-l+1\\j\neq i}}^{i+l-1}\!\!(1-\abs{y_{j}})\bigg)\big(f(\yb-\delta_{i-l}+2\delta_i-\delta_{i+l})-f(\yb)\big).
\end{aligned}
\]
In particular, if $b_{\bll,l}(\bll)\equiv0$ holds above an index $L>1$ then the corresponding process is said to be a \emp{parity conserving $L$-range branching and annihilating random walk}. Notice that the branching is ``careful'' in the sense that a particle branches to distance $l$ only if there are no other particles sitting on its $l-1$ long neighborhood. The state space $\Sc$ is therefore kept by the dynamics.

As it was the case earlier we consider the height function process $(\Xb^{(\mathrm{LR})}(t))_{t\geq0}$ corresponding to $(\Yb^{(\mathrm{LR})}(t))_{t\geq0}$. Now it involves long, but finite, rearrangements of bits on the same configuration space $\Sc_{\hgt}$. Its (formal) infinitesimal generator reads as:
\[
\Gr_{\hgt}^{(\mathrm{LR})}=\Grh+\sum_{l=2}^{+\infty}\Gr_{\hgt}^{\exl},
\]
where $\Grh$ was defined in \eqref{swappingvotergenerator}. On a local function $f$ the $l$-exclusion generator $\Grh^{\exl}$ acts as:
\begin{multline*}
 (\Grh^{\exl}f)(\xb)=\\
 \begin{aligned}
  \sum_{k\in\Sbb}&b_{k-\oh,l}^{\opl}(\yb)
  \bigg(\prod_{m=k}^{k+l-1}x_m(1-x_{m-l})\bigg)
  \bigg(f\big(\xb+\sum_{m=k}^{k+l-1}(\delta_{m-l}-\delta_m)\big)-f(\xb)\bigg)\\
  +&b_{k-\oh,l}^{\omi}(\yb)
  \bigg(\prod_{m=k}^{k+l-1}(1-x_m)x_{m-l}\bigg)
  \bigg(f\big(\xb+\sum_{m=k}^{k+l-1}(\delta_m-\delta_{m-l})\big)-f(\xb)\bigg),
 \end{aligned}
\end{multline*}
where $\xb\in\Sc_{\hgt}$, $\yb\in\Sc$ are connected by \eqref{spatderiv}, and the products and sums run over the half integers. The dynamics involves simultaneous swaps of blocks of length $l$, see Figure \ref{fig:dbarw-lrbarw}.
\begin{figure}[!ht]
\centering
\includegraphics[scale=0.3]{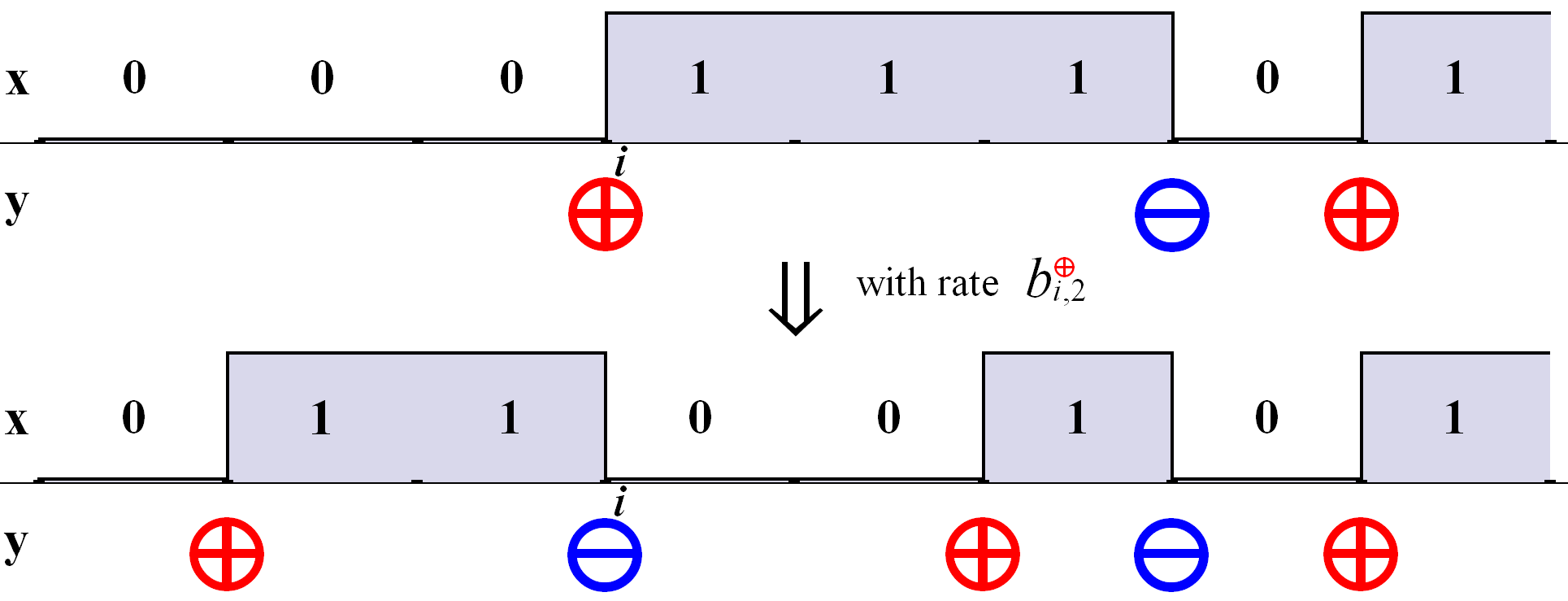}%
\caption[Long range branching and annihilating random walk]
{Long range branching and annihilating random walk branching at $i$.}\label{fig:dbarw-lrbarw}%
\end{figure}
The main result of this section is:
\begin{theorem}\label{non-nearesttightness}
Suppose all conditions of Theorem \ref{tightness}, that is assumptions \ref{A0}--\ref{A4} and that $\alpha_1\ul{s}>2\alpha_2\bar{d}$ holds, with the restriction that the function $B_{\bll}$, defined in \ref{A2}, is bounded above by a constant $\bar{B}$. We assume that for every fixed $2\leq l\in\Nbb$: $b_{i,l}:=b_{i,l}^{\opl}=b_{i,l}^{\omi}$ holds, and that $b_{\bll,l}(\bll)$ is translation invariant, that is $b_{i,l}(\yb)=b_{i+n,l}(\zb)$ whenever $z_j=y_{j-n}$ ($j,n\in\Zbb$). Furthermore we also assume that the $b_{\bll,l}(\bll)$'s fulfill assumption \ref{A3}.
The ultimate condition reads as:
{\rm
\begin{AAA}[resume]
\item\label{A5} \emp{Strength of long range branching}.\\
 There exists a function $\wtild{B}:\Nbb\to\Rbb^+$ such that for every $2\leq l\in \Nbb$:
 \[
 \sup_{i\in\Zbb,\yb\in\Sc}b_{i,l}(\yb) \leq \wtild{B}(l),\text{ and }\sum_{l=2}^{+\infty}l^2\wtild{B}(l)<+\infty
 \]
 hold.
\end{AAA}
}
\noindent Then we have the following consequences:
\begin{enumerate}
\item The process $(\Yb^{(\mathrm{LR})}(t))_{t\geq0}$ takes almost surely finitely many steps in any finite interval, and its width process has finite second moment for every fixed $t\geq0$.
\item Parity conserving long range branching and annihilating random walk as seen from its leftmost particle position is a countable state irreducible and ergodic Markov process.
\item In addition the particle number is of finite expectation at equilibrium and \eqref{widthestimate} also holds for its width process.
\end{enumerate}
\end{theorem}

\section{Particular models}\label{sec:partmodels}

This section is devoted to present particular but sufficiently general models the rates of which fulfill all the conditions in \ref{A0}--\ref{A4} or in \ref{A0}--\ref{A5}.
To handle both the case $\ch(\yb)=+1$ and when $\ch(\yb)=-1$, we introduce two sorts of indexing, mainly due to \ref{A3}. Let $i_1>i_2>\cdots>i_{\abs{\yb}}$ and $i_1<i_2<\cdots<i_{\abs{\yb}}$ denote the particle positions of $\yb$, if $\ch(\yb)=+1$ and $\ch(\yb)=-1$, respectively. In the following, we will also refer to these indices as \emp{ranks}. So we can write an arbitrary $\yb\in\Sc$ as:
\begin{align*}
 \yb=&\ind\{\ch(\yb)=+1\}\big(\delta_{i_{\abs{\yb}}}-\delta_{i_{\abs{\yb}-1}}+\delta_{i_{\abs{\yb}-2}}-
      \cdots+\delta_{i_{1}}\big)\\
    +&\ind\{\ch(\yb)=-1\}\big(-\delta_{i_1}+\delta_{i_2}-\delta_{i_3}+\cdots-\delta_{i_{\abs{\yb}}}\big),
\end{align*}
where recall the Kronecker delta function $\delta_{i_j}$ as $\delta_{i_j}(x)=1$, if $x=i_j$, otherwise it is zero.
In particular, when $\ch(\yb)=+1$, then the first particle from left to right order is indexed by $i_{\abs{\yb}}$, while the index of the last one in that order is $i_1$. This sort of reverse order indexing, indeed, facilitates the handling of condition \ref{A3}.

Notice that none of our conditions connect the behavior of branching and hopping rates (\eqref{assump:3decay} in \ref{A3} could be divided into three separate conditions as well). Hence we discuss these instances in two separate subsections and any two of the below rates for branching and random walking can be combined later to give rise to appropriate models.

Below, $\circ$ can stand for each of the charges $\opl$ and $\omi$.

\subsection{Branching rates}\label{sec:brrates}

It is clear that we have much greater freedom in choosing the branching rates than the hopping ones. All of the following rates satisfy \ref{A0}--\ref{A4} concerning branching and one can choose $\alpha_1,\alpha_2$ for $\alpha_1\ul{s}>2\alpha_2\bar{d}$ to hold as well, while the rates of the last item additionally satisfy \ref{A5} and $B_{\bll}\leq\bar{B}$ with some $\bar{B}>0$. Below we will denote by $\beta,\beta_1,\beta_2$ fixed, positive constants.
\begin{description}
\item[Finite range branching.] \emp{Any} translation invariant, positive and finite range dependent rates satisfy the corresponding assumptions for branching. The general setup can be the following:
\[
b_{i_j}^{\circ}(\yb)=\beta_1+\beta_2\!\!\sum_{i:\abs{i_j-i} \leq L_1}\!\!f\big(\abs{i_j-i},(y_{i'})_{\abs{i'-i}\leq L_2}\big)\quad(\yb\in\Sc),
\]
where $L_1,L_2\in\Nbb$ are fixed constants and $f:\Zbb\cap[0,L_1]\times\{-1,0,+1\}^{2L_2+1}\to\Rbb^+_0$ can be an arbitrary function. In particular this includes
\[
b_{i_j}^{\circ}(\yb)=\beta_1+\beta_2\ind\{y_{i_j-1}=0=y_{i_j+1}\},
\]
which gives boost on a lone particle to branch with bigger probability. Another choice that depends on the sign of the particle at position $i_j$:
\[
b_{i_j}^{\circ}(\yb)=\exp\big(y_{i_j}\textstyle\sum_{i:\abs{i_j-i} \leq L,\, i\neq i_j}f(\abs{i_j-i},y_{i})\big)\quad(\yb\in\Sc),
\]
where $L\in\Nbb$ and $f:\Zbb\cap[0,L]\times\{-1,0,+1\}\to\Rbb$ can be arbitrary.
\item[Bounded, non-finite range dependent branching.] We have great freedom in choosing bounded, translation invariant but \emp{non-finite range dependent} rates in a rather general way.
\begin{itemize}
\item The first branching rate instance of the previous point works in this case, if $L_1$ is set to be infinity along with
\[
\sum_{n=0}^{+\infty}\textstyle\max_{\xi\in\{-1,0,+1\}^{2L_2+1}} f(n,\xi)<+\infty.
\]
\item The next instance can serve as the general setup for non-finite range, distance- and configuration dependent branchings:
\[
b_{i_j}^{\circ}(\yb)=
\beta_1+\beta_2\,h\bigg(\sum_{i\in\Zbb}\psi(i_j-i)
f_{\abs{i_j-i}}\big((y_{i'})_{\abs{i'-i}\leq\abs{i_j-i}+L}\big)\bigg)\quad(\yb\in\Sc),
\]
where $h:\Rbb\to\Rbb^+$ is any $\gamma$-H\"older continuous function with positive exponent, $\psi:\Zbb\to\Rbb$ is absolutely summable, finally for an $n,L\in\Nbb$, $f_n:\{-1,0,+1\}^{2n+2L+1}\to\Rbb^+_0$ is any function for which
\[
\textstyle\sup_{n\in\Nbb,\,\xi\in\{-1,0,+1\}^{2n+2L+1}}f_n(\xi)<+\infty.
\]
Notice that, apart from the summability condition, we do not have any restriction on the order of decay of $\psi$.
\item Another model works with branching rates
\[
b_{i_j}^{\circ}(\yb)=
\beta\bigg(1+\sum_{i\in\Zbb}\psi(\abs{i_j-i})|y_{i}|\bigg)^{-\beta_1-\beta_2\lambda_{i_j}(\yb)}\quad(\yb\in\Sc),
\]
where $\lambda_{i_j}(\yb)=\sum_{i\in\Zbb} E(\abs{i_j-i})|y_i|$, $\psi,E:\Nbb\to\Rbb^+$ such that $\sum_{n=0}^{+\infty}\psi(n)<+\infty$ and $\lim_{n\to+\infty}E(n)=0$ hold.
\item Our next example involves a kind of \emp{rank} dependence:
\begin{align*}
b_{i_j}^{\circ}(\yb)=\beta + &\beta_1\sum_{j':1\leq j'\leq j}\; g_1(\abs{i_j-i_{j'}}, j')\\
+&\beta_2\sum_{j':j\leq j'\leq \abs{\yb}}\! g_2(\abs{i_j-i_{j'}},j')\qquad(\yb\in\Sc),
\end{align*}
where $g_1,g_2:\Nbb^2\to\Rbb$ are assumed to be absolutely summable sequences, that is $\sum_{n,m=0}^{+\infty}\big(\abs{g_1(n,m)}+\abs{g_2(n,m)}\big)<+\infty$, and the constant $\beta$ makes these rates strictly positive. In particular, when $g_1(n,m)=g_2(n,m)=m^{-2}\ind\{n=0,m>0\}$, then in terms of the height function $\Xb$, lone unit heights will repel each other.
\end{itemize}
\item[Unbounded, non-finite range branching.] Branching rates can even be chosen in a rather general way to be \emp{unbounded}. One only has to make sure that \eqref{assump:bdrift}, \eqref{assump:rateB} [see \ref{A2}] and \ref{A3} are all satisfied.
\begin{itemize}
\item Our first choice is the following:
\[
b_{i_j}^{\circ}(\yb) = \sum_{i\in\Zbb}\psi(\abs{i_j-i})|y_i|
f\big((y_{i'})_{\abs{i'-i}\leq L}\big)\quad(\yb\in\Sc),
\]
where $\psi:\Nbb\to\Rbb^+$ is monotone non-increasing such that
$\sum_{n:n\leq N}\psi(n) \leq \log(N)$ holds if $N$ is large enough, while $L\in\Nbb$ and $f:\{-1,0,+1\}^{2L+1}\to\Rbb^+$ can be any function. In particular
\[
\psi(n)=\frac{1}{n\log(n)\log(\log(n))\,\cdots\,\log^{(a)}(n)}\quad (\Zbb\ni n>\exp^{(a)}(0))
\]
serve as good choices for every $a\in\Nbb$, where $\log^{(a)}$ ($\exp^{(a)}$) denotes the $a^{\mathrm{th}}$ iterate of $\log$ ($\exp$).
\item Branching rates can be governed by a one-sided potential as well:
\[
b_{i_j}^{\circ}(\yb)=h\bigg(\sum_{1\leq j' \leq j}\varphi(\abs{i_j-i_{j'}})\bigg)\quad(\yb\in\Sc),
\]
where $h:[1,+\infty)\to\Rbb^+$ can be any monotone function with at most logarithmic growth, while $\varphi:\Nbb\to\Rbb^+$ is a monotone non-decreasing function such that $\varphi(0)=1$.
\item The next example is
\[
b_{i_j}^{\circ}(\yb)=
\beta\bigg(\sum_{i\in\Zbb}\psi(i_j-i)|y_{i}|\bigg)^{-y_{i_j}\lambda_{i_j}(\yb)}\quad(\yb\in\Sc),
\]
where $\psi(0)=1$, $\lambda_{i_j}(\yb)=\ch(\yb)\sum_{i\in\Zbb} E(i_j-i)|y_i|$, and $\psi,E:\Zbb\to\Rbb^+$.
Furthermore $\lim_{\abs{n}\to+\infty}\psi(n)=\lim_{\abs{n}\to+\infty}E(n)=0$, and
\[
\textstyle\sum_{n:\abs{n}\leq N}E(n)\leq\displaystyle\frac{\log(\log(N))}{\log\big(\sum_{n:\abs{n}\leq N}\psi(n)\big)}
\]
holds for large $N$'s. This instance allows for particles with charge $-\ch(\yb)$ to branch with unbounded rates.
\item We also have rank dependent examples:
\[
b_{i_j}^{\circ}(\yb)=\beta + \beta_1\, g_1(j) + \beta_2\, g_2(j)\sum_{j':1\leq j'\leq j}\varphi(\abs{i_{j}-i_{j'}})\quad(\yb\in\Sc),
\]
where $g_1:\Zbb^+\to\Rbb^+$ can be any function with uniformly bounded increments but at most logarithmic growth, $g_2:\Zbb^+\to\Rbb^+$, $\varphi:\Nbb\to\Rbb^+_0$ are monotone non-increasing functions such that $\sum_{n:n\leq N}\varphi(n)\leq\frac{\log(N)}{g_2(N)}$ holds, if $N$ is large enough.

In particular, when $g_1=\log$ and $\beta_2=0$, then most of the particles will branch roughly with rate $\log(\abs{\yb})$.
\end{itemize}
\item[Long range branching.] Finally we outline some instances for long range branching rates, where the only restrictions, compared to the above, are boundedness and a decay coming from \ref{A5}.
    Let $f:\Zbb\times\{-1,0,+1\}\to\Rbb^+$ be a bounded function, then a branching rate instance can be the following:
\[
b_{i_j,l}(\yb)=\beta_1\exp(-l)\!\sum_{i:\abs{i_j-i} \leq L}\!f\big(i_j-i,y_i\big)
+\beta_2\sum_{i\in\Zbb}\psi_l(i_j-i)|y_{i}|\quad(\yb\in\Sc),
\]
where $L\in\Zbb^+$, and $C, \varepsilon\in\Rbb^+$ for which $\sum_{n\in\Zbb}\psi_l(n) < C\, l^{-3-\varepsilon}$ holds for every $l\geq1$.
\end{description}

\subsection{Random walk rates}

The bottleneck of choosing the hopping rates is unequivocally assumption \ref{A4}. In general, those choices for $p$ and $q$ (see \eqref{ratepq}) will satisfy the hopping conditions of \ref{A0}, \ref{A1}, \ref{A3} and \ref{A4a} for which the (drift) function defined as
\begin{equation*}
c_k(\xb) := p_k(\xb)-q_k(\xb)\qquad(k\in\Sbb)
\end{equation*}
is monotone decreasing in $k$ for every fixed $\xb\in\Sc_{\hgt}$, in addition
\begin{align*}
0<&\inf_{k\in\Sbb,\xb\in\Sc_{\hgt}:\,x_{k-1}\neq x_{k}}\min\{q_k(\xb),c_k(\xb)+q_k(\xb)\},\text{ and}\\
&\sup_{k\in\Sbb,\xb\in\Sc_{\hgt}:\,x_{k-1}\neq x_{k}}\max\{q_k(\xb),c_k(\xb)+q_k(\xb)\}<+\infty
\end{align*}
are satisfied.

We will denote by $L_j$ the $j^{\mathrm{th}}$ inter-particle distance, that is $|i_{j+1}-i_j|$,
where $0\leq j\leq \abs{\yb}$, and formally add $L_0=L_{\abs{\yb}}:=+\infty$. In the following examples one can choose $\alpha_1,\alpha_2$ for $\alpha_1\ul{s}>2\alpha_2\bar{d}$ to hold, naturally.

\begin{description}
\item[{\bf Diffusive models.}]
\leavevmode
\begin{itemize}
\item Zero drift case.
\[
r_{i_j}^{\circ}(\yb)=\ell_{i_j}^{\circ}(\yb)\quad(\yb\in\Sc),
\]
where $r^{\circ}$ can be \emph{any} translation invariant, finite range dependent rate. Also it can be long range dependent one as well, satisfying assumption \ref{A3}. For instance:
\[
r_{i_j}^{\circ}(\yb)=\alpha\sum_{i\in\Zbb}|y_i|\abs{i_j-i+1}^{-\gamma} = \ell_{i_j}^{\circ}(\yb)\quad (\yb\in\Sc),
\]
where $0<\alpha$ and $1<\gamma$.
\item Constant drifting models. The general setup reads as:
\begin{equation*}
r_{i_j}^{\circ}(\yb)=\frac{f_{i_j}(\yb)}{Z_{i_j}(\yb)};\quad
\ell_{i_j}^{\circ}(\yb)=\frac{g_{i_j}(\yb)}{Z_{i_j}(\yb)}\quad(\yb\in\Sc),
\end{equation*}
where $Z_{i_j}(\yb)=f_{i_j}(\yb)-g_{i_j}(\yb)$, such that $\sup_{i\in\Zbb,\yb\in\Sc}g_{i}(\yb)<\inf_{i\in\Zbb,\yb\in\Sc}f_{i}(\yb)$. Now, $f,g:\Zbb\times\Sc\to\Rbb^+$ can be any bounded and finite range dependent functions. Also these can be non-finite range dependent ones, in particular
\[
f_{i_j}^{\circ}(\yb)=\beta + \sum_{i:i > i_j}|y_i|\abs{i_j-i}^{-\gamma},
\text{ and }\;g_{i_j}^{\circ}(\yb)=\sum_{i:i < i_j}|y_i|\abs{i_j-i}^{-\gamma},
\]
where $\gamma>1$ and $\beta$ makes $f$ be above $g$.
\end{itemize}
\item[{\bf Rank dependent models.}] Recall that the indexing \(i_j\) increases from right to left for \(\ch(\yb)=+1\) configurations.
\begin{itemize}
\item First,
\[
r_{i_j}^{\circ}(\yb)=g(j)h(\abs{i_1-i_j})=1-\ell_{i_j}^{\circ}(\yb)\quad (\yb\in\Sc),
\]
where both $g,h:\Zbb^+\to(\varepsilon,1-\varepsilon)$ are monotone non-decreasing (non-increasing) functions and $0<\varepsilon<1$, if $\ch(\yb)=+1$ ($\ch(\yb)=-1$).

In particular in case of $\ch(\yb)=-1$ and $g(n)=h(n)=n^{-1}$ ($n\in\Zbb^+$), particles will hop to the right with rate $\approx w(\yb)^{-1}$ at low density, while with rate $\approx\abs{\yb}^{-2}$ at high density.
\item Our second example is the following:
\begin{align*}
r_{i_j}^{\circ}(\yb) &= \alpha - \ch(\yb)\sum_{j':j<j'\leq \abs{\yb}}\frac{1}{(j')^{\gamma}}
\sum_{\Zbb\ni i\neq i_j}\psi(\abs{i_{j'}-i})|y_{i}|; \\
\ell_{i_j}^{\circ}(\yb) &= \alpha\pm\ch(\yb)\sum_{j':1\leq j'<j}\;\;\frac{1}{(j')^{\gamma}}
\sum_{\Zbb\ni i\neq i_j}\psi(\abs{i_{j'}-i})|y_{i}| \quad(\yb\in\Sc),
\end{align*}
where $\gamma>1$, $\psi:\Nbb\to\Rbb^+$ is such that $\sum_{n\in\Nbb}\psi(n)<+\infty$ and $\alpha>0$ makes the rates positive. The empty sum is, as usual, defined to be zero.
\end{itemize}
\item[{\bf Distance dependent models.}]
\leavevmode
\begin{itemize}
\item The general setup can be enclosed in the following example:
\begin{align*}
r_{i_j}^{\circ}(\yb) = h\bigg(\sum_{1\leq j'\leq j} \varphi(\abs{i_j-i_{j'}})\bigg)=1-\ell_{i_j}^{\circ}(\yb) \quad(\yb\in\Sc),
\end{align*}
where $\varphi:\Nbb\to\Rbb$ is monotone non-decreasing function and if $\ch(\yb)=+1$, then  $h:\Rbb\to(\varepsilon,1-\varepsilon)$ is also monotone non-decreasing, otherwise (if $\ch(\yb)=-1$) it is a monotone non-increasing function with $0<\varepsilon<1$. In this model a strong potential ($\varphi$) forces (distant) particles to move towards each other in such a way that the strength of the interaction grows in distance.
\item Replacing $\varphi$, in the previous example, with the inter-particle distances we get the following:
\begin{align*}
r_{i_j}^{\circ}(\yb) = h\bigg(\sum_{1\leq j'< j} g(L_{j'})\bigg)=1-\ell_{i_j}^{\circ}(\yb) \quad(\yb\in\Sc),
\end{align*}
where this time $g:\Zbb^+\to\Rbb^+$ can be any function (such that $g(+\infty):=0$). This instance resembles the above rank dependent examples.
\item The following example weakens the interaction between (distant) particles:
\begin{align*}
r_{i_j}^{\circ}(\yb) &= \alpha + \ch(\yb)\sum_{j':1< j'\leq j}\psi(\abs{i_{j'}-i_1});\\
\ell_{i_j}^{\circ}(\yb) &= \alpha - \ch(\yb)\sum_{j':1< j'\leq j}\psi(\abs{i_{j'}-i_1})\quad(\yb\in\Sc),
\end{align*}
where $\psi:\Nbb\to\Rbb^+$ is such that $\sum_{n\in\Nbb}\psi(n)<+\infty$ and $\alpha>0$ makes these rates positive. In this case a potential is placed on the first particle, according to ranks, with which it tries to pull the other particles towards itself.
\item Another example works with the inter-particle distances (and ranks):
\begin{align*}
r_{i_j}^{\opl}(\yb)&=\oh+\ind\{\ch(\yb)=+1\}g(L_{j-1},j-1)-\ind\{\ch(\yb)=-1\}g(L_{j},j)\\
&=1-\ell_{i_j}^{\opl}(\yb);\\
r_{i_j}^{\omi}(\yb)&=\oh+
\ind\{\ch(\yb)=+1\}g(L_{j},j)-\ind\{\ch(\yb)=-1\}g(L_{j-1},j-1)\\
&=1-\ell_{i_j}^{\omi}(\yb),
\end{align*}
where $\yb\in\Sc$, $g:\Zbb^+\times\Zbb^+\to(\varepsilon,\oh-\varepsilon)$ is any function and $0<\varepsilon<\oh$. We set $g(+\infty,\bll)$ to be $0$. Notice that, if $g(L,n)$ does not depend on $L\in\Zbb^+$ and is monotone decreasing (increasing) in $n\in\Zbb^+$, when $\ch(\yb)=+1$ ($\ch(\yb)=-1$), then the interaction becomes repulsive between consecutive $\omi$ and $\opl$ particles.
\item Our last two examples involve long-range attraction between consecutive particles of charges $\opl$ and $\omi$. First,
\begin{align*}
r_{i_j}^{\opl}(\yb)&=\oh+\textstyle\sum_{j + (1-\ch(\yb))/2\, < j'\leq \abs{\yb}}\psi(\abs{i_{j}-i_{j'}})=1-\ell_{i_j}^{\opl}(\yb);\\
r_{i_j}^{\omi}(\yb)&=\oh+\textstyle\sum_{j + (1+\ch(\yb))/2\, < j'\leq \abs{\yb}}\psi(\abs{i_{j}-i_{j'}})=1-\ell_{i_j}^{\omi}(\yb),
\end{align*}
such that we fix $r_{i_{\mathrm{right}}}^{\circ}(\yb)=r_{i_{\mathrm{left}}}^{\circ}(\yb)=\oh=
\ell_{i_{\mathrm{right}}}^{\circ}(\yb)=\ell_{i_{\mathrm{left}}}^{\circ}(\yb)$, where $\yb\in\Sc$, $\psi:\Nbb\to\Rbb^+$ is monotone decreasing and $\sum_{n\in\Nbb}\psi(n)<\oh$.
\item Finally,
\begin{align*}
r_{i_j}^{\opl}(\yb)=\oh + & \ind\{\ch(\yb)=+1\}\sum_{\Zbb\ni i\neq i_j}\psi\big(\big|(i_{j-1}+i_{j})/2-i\big|\big)|y_i|\\
-&\ind\{\ch(\yb)=-1\}\sum_{\Zbb\ni i\neq i_j}\psi\big(\big|(i_{j}+i_{j+1})/2-i\big|\big)|y_i|=1-\ell_{i_j}^{\opl}(\yb);\\
r_{i_j}^{\omi}(\yb)=\oh + & \ind\{\ch(\yb)=+1\}\sum_{\Zbb\ni i\neq i_j}\psi\big(\big|(i_{j}+i_{j+1})/2-i\big|\big)|y_i|\\
-&\ind\{\ch(\yb)=-1\}\sum_{\Zbb\ni i\neq i_j}\psi\big(\big|(i_{j-1}+i_{j})/2-i\big|\big)|y_i|=1-\ell_{i_j}^{\omi}(\yb),
\end{align*}
where $\yb\in\Sc$, $i_{0}:=-\infty$, while $i_{\abs{\yb}+1}:=+\infty$ and $\psi:\Zbb^+_0\cup\,\Sbb^+\to\Rbb^+_0$ is such that $\sum_{n\in\Zbb^+\cup\;\Sbb^+}\psi(n)<\oh$.
\end{itemize}
\end{description}
We notice that the last three examples satisfy part \ref{A4b} of assumption \ref{A4}. Also notice that the last two instances above almost correspond to a natural attraction potential. Unfortunately the twists we used there are needed for our proofs to work.

\section{Proofs}\label{sec:proof}

The first two proofs correspond to the assertions of Subsection \ref{sec:construction}.
\begin{proof}[Proof of Proposition \ref{propwelldefined}.]
First define the total rate function corresponding to a configuration $\yb\in\Sc$ by
\[
 R(\yb)=\textstyle\sum_{i\in\Zbb}\abs{y_i}\cdot\left\{\alpha_1(r_i^{\opl}(\yb)+\ell_i^{\opl}(\yb)+r_i^{\omi}(\yb)+
\ell_i^{\omi}(\yb))+\alpha_2(b_i^{\opl}(\yb)+b_i^{\omi}(\yb))\right\},
\]
which is initially finite by definition. Let $N(t)$ be the number of steps up to time $t$ that were taken by the process $(\Yb(t))_{t\geq0}$ of double branching annihilating random walkers. We will prove that $N(t)<+\infty$ a.s.\ for every $t>0$. First we define $\big(V_n\big)_{n\in\Nbb}$ to be a sequence of independent exponentially distributed random variables with mean
\begin{equation*}
\wtild{m}_n:=\Exp(V_n)=\frac{1}{\alpha_1(\abs{\Yb_0}+2n)+\alpha_2B(\abs{\Yb_0}+2n)},
\end{equation*}
where $\Yb_0\in\Sc$ is an arbitrary initial configuration, $n\geq1$ and $V_0=0$. Now, let the process $\wtild{N}$ be defined as
\begin{equation*}
\wtild{N}(t)=\abs{\Yb_0}+\min\{n>0:S_n\geq t\}\quad(t>0),\text{ where }S_n=V_0+V_1+V_2+\cdots+V_n.
\end{equation*}
Due to $R(\yb)\leq \max_{1\leq n\leq N}(\alpha_1 n + \alpha_2 nB_n)\leq \alpha_1N+\alpha_2 B(N)$, if $\abs{\yb}\leq N$ by an obvious coupling it follows that $N(t)\leq\wtild{N}(t)$ holds, that is
\[
\wtild{\tau}_n:=\inf\{t>0:\wtild{N}(t)=n\}\leq \inf\{t>0:N(t)=n\}=:\tau_n\quad (n\in\Nbb),
\]
where assumptions \eqref{assump:R21b} [\ref{A1}] and \eqref{assump:R22b} [\ref{A2}] are involved.\\
Hence it is enough to prove that $\lim_{n\to+\infty}\wtild{\tau}_{n}=+\infty$ which is equivalent to claiming that
\[
S_\infty:=\lim_{n\to+\infty}S_n=+\infty\mbox{\; a.s.\ }
\]
Notice that the almost sure limit exists by monotonicity and that $\sum_{j=1}^{+\infty}\wtild{m}_j^2<+\infty$, which latter fact implies that the process $(M_n:=S_n-\sum_{j=1}^n \wtild{m}_j)_{n\in\Nbb}$ has an almost sure finite limit $M_{\infty}$ as $n\to+\infty$ by Doob's martingale convergence theorem. Now, assume for contradiction that $S_\infty<+\infty$ holds with positive probability, then it would follow that $M_{\infty}=-\infty$ with positive probability leading to a contradiction, since by assumption \eqref{assump:rateBa}: $\sum_{j=1}^{+\infty} \wtild{m}_j=+\infty$. We finished the proof.
\end{proof}
\begin{definition}[maximum width process]\label{defmaxwidthprocess}
Let $\left(U_n\right)_{n\in\Nbb}$ be a sequence of independent exponentially distributed random variables with mean
\begin{equation}
\Exp U_n=\frac{1}{K(w_0+n)}\quad (n\in\Zbb^+),\label{Qszelbecslese}
\end{equation}
where the constant $K\in\mathbb{N}$ will be chosen appropriately later. Under assumption \eqref{assump:rateBb} [from \ref{A2}] one can choose an $N_0$ large enough for
\[
\max_{1-w_0\leq i\leq n}B_{w_0+i} \leq \bar{D}(w_0+n),
\]
to hold whenever $n>N_0\geq1$. We choose $K$ to be
\[
K>\max\big\{2\alpha_1+2\alpha_2\max_{1\leq i\leq N_0}(B_i),\,2\alpha_1+2\alpha_2 \bar{D}\big\}.
\]
Finally let
\[
Q^{w_0}(t)=w_0+\min\{n>0:U_1+U_2+\cdots+U_n\geq t\}\quad(t>0),
\]
where $w_0=W(0)=w(\Yb_0)\in\Zbb^+$.
\end{definition}
We will use these notations and estimates in the upcoming proof.
\begin{proof}[Proof of Proposition \ref{propvariance}.]
Notice that in Definition \ref{defmaxwidthprocess} assumptions \eqref{assump:R22b} and \eqref{assump:rateBb} from \ref{A2} were implicitly used. Now under assumption \eqref{assump:R21b} it follows that one can set up a coupling such that $W(t)\leq Q^{w_0}(t)$ holds whenever $W(0)=w_0$ and $t\geq0$. It implies that the $r^\mathrm{th}$ moment of $\abs{\Yb(t)}$ is at most
\begin{align}
\Exp\abs{\Yb(t)}^r
&\leq\Exp W(t)^r \leq \Exp Q^{w_0}(t)^{\uip{r}} \leq \sum_{n=1}^{+\infty}\Prob(Q^{w_0}(t)\geq \lfloor n^{1/\uip{r}}\rfloor)\nonumber\\
&\leq (w_0+1)^{\uip{r}}+\sum_{n=(w_0+1)^{\uip{r}}}^{+\infty}\!\!\Prob(U_{w_0+1}+\cdots+U_{\lfloor n^{1/\uip{r}}\rfloor}<t),
\label{propvarbecsles}
\end{align}
where $r>0$ and $\lceil x\rceil$ ($\lfloor x\rfloor$) denotes the upper (lower) integer part of $x\in\Rbb^+$. For simplicity let $r_n:=\lfloor n^{1/\uip{r}}\rfloor$. We will estimate the sum in the last display \eqref{propvarbecsles} using Markov's inequality. For a fixed $\lambda>0$:
\begin{align}
\Prob(U_{w_0+1}+\cdots+U_{r_n}<t)
&\leq\exp(\lambda t)\prod_{j=w_0+1}^{r_n}\Exp(\exp(-\lambda U_j))\nonumber\\
&=\exp\bigg(\lambda t+\sum_{j=w_0+1}^{r_n}\log\bigg(1-\frac{\lambda}{m_j^{-1}+\lambda}\bigg)\bigg)\nonumber\\
&\leq\exp\bigg(\lambda t-\lambda\sum_{j=w_0+1}^{r_n}\frac{1}{2 K j+\lambda}\bigg)\nonumber\\
&\leq\exp\big(\lambda t+\lambda\log(2 K w_0+\lambda)\big)\frac{1}{\big(2 K r_n+\lambda\big)^{\lambda/(2 K)}},\label{prop2:turbomarkov}
\end{align}
where $m_j=\Exp(U_j)$ which is defined in \eqref{Qszelbecslese}, hence $m_j^{-1} < 2 K j$ if $j>w_0$. Finally choosing $\lambda$ to be strictly greater than $2K\uip{r}$ we find that \eqref{prop2:turbomarkov} is summable in $n$, hence the proof is complete.
\end{proof}

\bigskip

Next we define the number of the so-called inversions, or in other words ``wrongly ordered pairs''. This function is denoted by $f_{\mathrm{CD}}$ and is defined for a configuration $\xb\in\Sc_{\hgt}$ as
\[
f_{\mathrm{CD}}(\xb)=\sum_{k\in\Sbb}\sum_{l:l>k}\ind
\left\{x_k=\text{{\small$\oh$}}(1+\ch(\yb)),\,x_l=\text{{\small$\oh$}}(1-\ch(\yb))\right\}.
\]
This function was first used in \cite{CD95}, subsequently in \cite{BFMP01} and in \cite{SS08T}. In all these papers slightly different approaches were followed to prove positive recurrence in slightly different models. Our strategy will mainly follow the arguments of the last paper. As it was commented in \cite{BFMP01} $f_{\mathrm{CD}}$ measures the \emp{distance} from $\xb\in\wtild{\Sc}_{\hgt}$ to the Heaviside step configuration $(\ind\{\ch(\yb)k>0\})_{k\in\Sbb}$, since it counts exactly how many nearest neighbor transpositions are needed to reach the Heaviside from $\xb$.
\begin{lemma}\label{lemma}
Assume that \eqref{assump:R21b} [from \ref{A1}], \eqref{assump:bdrift}, \eqref{assump:R22b}, \eqref{assump:rateBb} [from \ref{A2}] and \ref{A4} hold. Then there exist $c,C\in\Rbb^+$ such that
\begin{equation}\label{lemmaestimate}
(\Grh f_{\mathrm{CD}})(\xb)\leq C - c \abs{\yb}
\end{equation}
holds for every $\xb\in\Sc_{\hgt}$ providing that $\alpha_1\ul{s}>2\alpha_2\bar{d}$. Recall that $\abs{\yb}$ is the particle number in configuration $\yb\in\Sc$, that is the number of phase boundaries $\sum_{l\in\Sbb}\ind\{x_l\neq x_{l+1}\}$ of $\xb$.
\end{lemma}
\begin{remark}
This lemma is of essential importance on the way of proving Theorem \ref{tightness}. In plain words it tells us that the number of inversions decays above a certain number of double branching annihilating random walkers (i.e.\ the right hand-side of \eqref{lemmaestimate} becomes negative).
\end{remark}
\begin{proof}[Proof of Lemma \ref{lemma}.]
We prove \eqref{lemmaestimate} by calculating the act of generator $\Grh$, given in \eqref{swappingvotergenerator}, on $f_{\mathrm{CD}}$ in two steps according to its different parts.\\
\noindent {\bf I (spin-flip part contribution).}
From now on, we fix a configuration $\xb\in\Sc_{\hgt}$, the interface of which is denoted by $\yb$. We introduce the notation $\kappa(\yb):=(1+\ch(\yb))/2\in\{0,1\}$ to compress the proof and we shorten $\ch(\yb)$ and $\kappa(\yb)$ to $\ch$ and $\kappa$, respectively.
First, notice that
\begin{equation*}
f_{\mathrm{CD}}\big(\xb+(1-2x_k)\delta_k\big)-f_{\mathrm{CD}}(\xb)
=\big(\ind\{x_k=\kappa\}-\ind\{x_{k}=1-\kappa\}\big)S(k),
\end{equation*}
where
\begin{equation}\label{eq:lemmaS}
S(k)=\sum_{l:l<k}\ind\{x_l=\kappa\}-\sum_{l:l>k}\ind\{x_l=1-\kappa\}.
\end{equation}
It is clear that both sums in the last display are finite.
Using \eqref{ratepq} and putting the above connection into \eqref{voterpartgenerator} we arrive to
\begin{align}
\big(\Grf f_{\mathrm{CD}}\big)(\xb)
=&\sum_{k\in\Sbb}
\left(p_k(\xb)+q_{k+1}(\xb)\right)\big(f_{\mathrm{CD}}\big(\xb+(1-2x_k)\delta_k\big)-f_{\mathrm{CD}}(\xb)\big)\nonumber\\
=&\sum_{k\in\Sbb}\big(\ind\{x_k=\kappa\}-\ind\{x_{k}=1-\kappa\}\big)(p_{k}(\xb)+q_{k+1}(\xb))S(k),\label{eq:votergeneratoreffect}
\end{align}
that is
\begin{align*}
\big(\Grf f_{\mathrm{CD}}\big)(\xb)
=&\sum_{k\in\Sbb}\sum_{l:l<k}\ind\{x_l=\kappa\}p_{k}(\xb)(\ind\{x_k=\kappa\}-\ind\{x_k=1-\kappa\})\\
+&\sum_{k\in\Sbb}\sum_{l:l>k}\ind\{x_l=1-\kappa\}p_{k}(\xb)(\ind\{x_k=1-\kappa\}-\ind\{x_k=\kappa\})\\
+&\sum_{k\in\Sbb}\sum_{l:l<k}\ind\{x_l=\kappa\}q_{k+1}(\xb)(\ind\{x_k=\kappa\}-\ind\{x_k=1-\kappa\})\\
+&\sum_{k\in\Sbb}\sum_{l:l>k}\ind\{x_l=1-\kappa\}q_{k+1}(\xb)(\ind\{x_k=1-\kappa\}-\ind\{x_k=\kappa\}).
\end{align*}
Now, changing the order of summations in the last display while using the identity $p_k(\xb)\ind\{x_{k-1}=x_k\}\equiv q_k(\xb)\ind\{x_{k-1}=x_k\}\equiv0$, we obtain the following:
\begin{align}
\big(&\Grf f_{\mathrm{CD}}\big)(\xb)\nonumber\\
=&\sum_{l\in\Sbb}\ind\{x_l=\kappa\}
\sum_{k:k>l}p_k(\xb)\big(\ind\{x_{k-1}=1-\kappa,x_k=\kappa\}-\ind\{x_{k-1}=\kappa,x_k=1-\kappa\}\big)\nonumber\\
+&\sum_{l\in\Sbb}\ind\{x_l=1-\kappa\}
\sum_{k:k<l}p_k(\xb)\big(\ind\{x_{k-1}=\kappa,x_k=1-\kappa\}-\ind\{x_{k-1}=1-\kappa,x_k=\kappa\}\big)\nonumber\\
+&\sum_{l\in\Sbb}\ind\{x_l=\kappa\}\sum_{k:k>l}
q_{k+1}(\xb)\big(\ind\{x_{k}=\kappa,x_{k+1}=1-\kappa\}-\ind\{x_{k}=1-\kappa,x_{k+1}=\kappa\}\big)\nonumber\\
+&\sum_{l\in\Sbb}\ind\{x_l=1-\kappa\}
\sum_{k:k<l}q_{k+1}(\xb)\big(\ind\{x_{k}=1-\kappa,x_{k+1}=\kappa\}-\ind\{x_{k}=\kappa,x_{k+1}=1-\kappa\}\big).
\label{eq:votergeneratoreffect2}
\end{align}
For sake of simplicity we introduce the following notation:
\[
\ind_{k-1,k}^{\kappa}(\xb):=\ind\{x_{k-1}=1-\kappa,x_k=\kappa\}-\ind\{x_{k-1}=\kappa,x_k=1-\kappa\}\quad(k\in\Sbb).
\]
Furthermore denote by $I_1$, $I_2$, $I_3$ and $I_4$ the first, second, third and fourth sum of \eqref{eq:votergeneratoreffect2}, respectively. Conditioning on whether there is a particle at $l+\oh$ and at $l-\oh$ or not, we get the following transformations for the sums $I_1$, $I_3$ and $I_2$, $I_4$, respectively:
\begin{align*}
I_1=&\sum_{l\in\Sbb}\ind\{x_{l}=\kappa,x_{l+1}=\kappa\}\cdot\sum_{k:k>l}p_{k+1}(\xb)\ind_{k,k+1}^{\kappa}(\xb)\\
+&\sum_{l\in\Sbb}\ind\{x_{l}=\kappa,x_{l+1}=1-\kappa\}\bigg(-p_{l+1}(\xb)+\sum_{k:k>l}p_{k+1}(\xb)\ind_{k,k+1}^{\kappa}(\xb)\bigg),\\
I_3=&\sum_{l\in\Sbb}\ind\{x_l=\kappa,x_{l+1}=\kappa\}\sum_{k:k>l}q_{k+1}(\xb)\ind_{k,k+1}^{1-\kappa}(\xb)\\
+&\sum_{l\in\Sbb}\ind\{x_l=\kappa,x_{l+1}=1-\kappa\}\bigg(q_{l+1}(\xb)+\sum_{k:k>l}q_{k+1}(\xb)\ind_{k,k+1}^{1-\kappa}(\xb)\bigg)\\
-&\sum_{l\in\Sbb}\ind\{x_l=\kappa,x_{l+1}=1-\kappa\}q_{l+1}(\xb),\text{ and}\\
I_2=&\sum_{l\in\Sbb}\ind\{x_{l-1}=1-\kappa,x_l=1-\kappa\}\sum_{k:k<l}p_k(\xb)\ind_{k-1,k}^{1-\kappa}(\xb)\\
+&\sum_{l\in\Sbb}\ind\{x_{l-1}=\kappa,x_l=1-\kappa\}\bigg(p_{l}(\xb)+\sum_{k:k<l}p_k(\xb)\ind_{k-1,k}^{1-\kappa}(\xb)\bigg)\\
-&\sum_{l\in\Sbb}\ind\{x_{l}=\kappa,x_{l+1}=1-\kappa\}p_{l+1}(\xb),\\
I_4=&\sum_{l\in\Sbb}\ind\{x_{l-1}=1-\kappa,x_l=1-\kappa\}\sum_{k:k<l}q_{k}(\xb)\ind_{k-1,k}^{\kappa}(\xb)\\
&\sum_{l\in\Sbb}\ind\{x_{l-1}=\kappa,x_l=1-\kappa\}\bigg(-q_{l}(\xb)+\sum_{k:k<l}q_{k}(\xb)\ind_{k-1,k}^{\kappa}(\xb)\bigg),
\end{align*}
where the inner sums in the second terms of $I_2$ and $I_3$ were (artificially) completed with $p_{l}(\xb)$ and $q_{l+1}(\xb)$, respectively.
Regrouping the terms of $I_1+I_2+I_3+I_4$ (that is of \eqref{eq:votergeneratoreffect2}), we obtain
\begin{align*}
\big(&\Grf f_{\mathrm{CD}}\big)(\xb)=\\
-&\sum_{l\in\Sbb}(p_{l+1}(\xb)+q_{l+1}(\xb))\ind\{x_l=\kappa,x_{l+1}=1-\kappa\}\\
+&\sum_{l\in\Sbb}\ind\{x_l=\kappa,x_{l+1}=\kappa\}\sum_{k:k>l}(p_{k+1}(\xb)-q_{k+1}(\xb))\ind_{k,k+1}^{\kappa}(\xb)\\
+&\sum_{l\in\Sbb}\ind\{x_l=\kappa,x_{l+1}=1-\kappa\}\bigg(q_{l+1}(\xb)-p_{l+1}(\xb)+\sum_{k:k>l}(p_{k+1}(\xb)-q_{k+1}(\xb))
\ind_{k,k+1}^{\kappa}(\xb)\bigg)\\
+&\sum_{l\in\Sbb}\ind\{x_{l-1}=1-\kappa,x_l=1-\kappa\}\sum_{k:k<l}(q_k(\xb)-p_k(\xb))\ind_{k-1,k}^{\kappa}(\xb)\\
+&\sum_{l\in\Sbb}\ind\{x_{l-1}=\kappa,x_l=1-\kappa\}\bigg(p_{l}(\xb)-q_{l}(\xb)+\sum_{k:k<l}(q_k(\xb)-p_k(\xb))
\ind_{k-1,k}^{\kappa}(\xb)\bigg),
\end{align*}
where we used that $\ind_{k-1,k}^{\kappa}(\xb)=-\ind_{k-1,k}^{1-\kappa}(\xb)$ holds for every $\xb\in\Sc_{\hgt}$ and $k\in\Sbb$.

Now, the first sum in the last display is clearly at most $-\oh\ul{s}\,(\abs{\yb}-1)$ by assumption \eqref{assump:R21a}.
The remaining terms are by now carefully prepared to match condition \ref{A4} and hence give a non-positive contribution. Indeed, in each of these terms the factors $\ind_{\bll,\bll}^\kappa(\xb)$ give alternate signs whenever the height $\xb$ changes from $1-\kappa$ to $\kappa$ or from $\kappa$ to $1-\kappa$, and the indicators for positions $l-1$ or $l+1$ in front start the parity of these $(1-\kappa)$-$\kappa$ or $\kappa$-$(1-\kappa)$ steps just the right way for assumption \ref{A4} to result in a non-positive contribution for each of these lines (see Figure \ref{fig:dbarw-proof-lemma} for an illustration of the heights). In case of \ref{A4b} we take advantage of the rightmost (leftmost) particle being more likely to jump towards, than against any particle, which is encompassed in \eqref{eq:supplA4}.
\begin{figure}[!ht]
\centering
\includegraphics[scale=0.3]{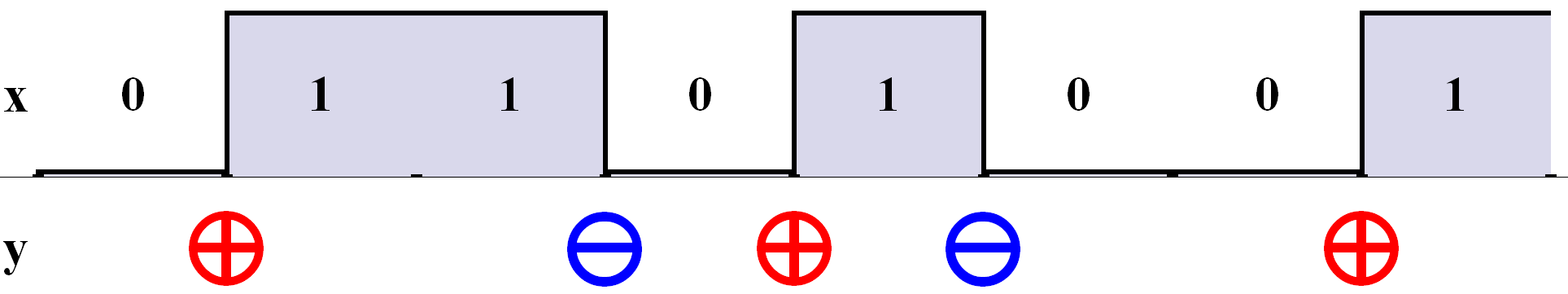}%
\caption[To the proof of Lemma \ref{lemma}.]
{Symmetries in the configurations for the proof of Lemma \ref{lemma}.}\label{fig:dbarw-proof-lemma}%
\end{figure}
Overall, we obtain that the spin-flip part contribution is at most:
\begin{equation}\label{voterpartcontr}
(\Grf f_{\mathrm{CD}})(\xb) \leq \ul{s}-\ul{s}\frac{\abs{\yb}}{2}.
\end{equation}

\noindent {\bf II (exclusion part contribution).}
It is easy to see that for a $k\in\Sbb$ we have
\begin{equation*}
\ind\{x_{k+1}-x_k=\pm1\}\big(f_{\mathrm{CD}}(\xb\pm \delta_{k}\mp \delta_{k+1})-f_{\mathrm{CD}}(\xb)\big)=\pm\ch.
\end{equation*}
Now plugging this into \eqref{exclusionpartgenerator} we arrive to
\begin{align}
(\Grex f_{\mathrm{CD}})(\xb)=\sum_{k\in\Sbb}
&\ind\{x_{k+1}-x_k=+1\}b_{k+\oh}^{\opl}(\yb)(f(\xb+\delta_{k}-\delta_{k+1})-f(\xb))\nonumber\\
+&\ind\{x_{k+1}-x_k=-1\}b_{k+\oh}^{\omi}(\yb)(f(\xb-\delta_{k}+\delta_{k+1})-f(\xb))\nonumber\\
=&\ch\cdot\left[\sum_{i\in\Zbb}b_{i}^{\opl}(\yb)-\sum_{j\in\Zbb}b_{j}^{\omi}(\yb)\right]\leq \bar{d}\abs{\yb},\label{swappingpartcontr}
\end{align}
due to assumption \eqref{assump:bdrift} from \ref{A2}. As before we used the connection: $y_i=x_{i+\oh}-x_{i-\oh}$ ($i\in\Zbb$).

Finally summarizing steps {\bf I} and {\bf II}, putting together \eqref{voterpartcontr} and \eqref{swappingpartcontr} we arrive to the following chain of inequalities:
\[
\begin{aligned}
(\Grh f_{\mathrm{CD}})(\xb)
&=\alpha_1\big(\Grf f_{\mathrm{CD}}\big)(\xb)+\alpha_2\big(\Grex f_{\mathrm{CD}}\big)(\xb)\\
&\leq \alpha_1\ul{s}-\big(\alpha_1\ul{s}-2\alpha_2\bar{d}\big)\frac{\abs{\yb}}{2} = C-c\abs{\yb}
\end{aligned}
\]
with $C=\alpha_1\ul{s}$ and $c=\oh\alpha_1\ul{s}-\alpha_2\bar{d}$, which are just assumed to be strictly positive.
\end{proof}
\begin{lemma}\label{lemmagenerator}
Under assumptions \eqref{assump:R21b} from \ref{A1}, \eqref{assump:R22b} and \eqref{assump:rateBb} from \ref{A2}, for every $t>0$ and $\Xb_0=\Xb(0)\in\Sc_{\hgt}$ we have
\begin{equation}\label{lemmamarti}
\Exp[w(\Yb(t))]\leq 2+f_{\mathrm{CD}}(\Xb_0)+\int_{0}^t\Exp\left[(\Grh f_{\mathrm{CD}})(\Xb(s))\right]\,\dr s<+\infty.
\end{equation}
\end{lemma}
\begin{remark}
It is clear by means of Proposition \ref{propvariance} that for every $t>0$ every moment of $f_{\mathrm{CD}}(\Xb(t))$ is finite. However, \emp{a priori} we cannot assert the same for $(\Grh f_{\mathrm{CD}})(\Xb(t))$. Under mild conditions this will be proved along with \eqref{lemmamarti} which in fact could also be proved from Lemma \ref{lemma} directly.
\end{remark}
\begin{proof}[Proof of Lemma \ref{lemmagenerator}.]
The proof relies on a standard martingale argument. The previously defined coupling and the proof of Proposition \ref{propvariance} yields that $Q^{w_0}(t)$ dominates $W(t)$ for all $t\geq0$. Setting up the following stopping times:
\[
\tau_Q^{(n)}:=\inf\{t\geq0:Q(t)>n\},\quad \tau_{W}^{(n)}:=\inf\{t\geq0:W(t)>n\},
\]
where $n\in\Zbb^+$, it is straightforward that the process:
\[
M^{(n)}(t):=f_{\mathrm{CD}}\big(\Xb\big(t\wedge \tau_{W}^{(n)}\big)\big)-\int_{0}^{t\wedge\tau_{W}^{(n)}}(\Grh f_{\mathrm{CD}})(\Xb(s))\,\dr s,
\]
is a continuous time martingale. Since $\tau_Q^{(n)} \leq \tau_{W}^{(n)}$ holds for all $n$, $\tau_{W}^{(n)}\to+\infty$ a.s.\ and we obtain that $M^{(n)}(t)\to M(t)$ a.s. as $n\to+\infty$, where
\[
M(t)=f_{\mathrm{CD}}(\Xb(t))-\int_{0}^{t}(\Grh f_{\mathrm{CD}})(\Xb(s))\,\dr s.
\]
Now recalling \eqref{eq:lemmaS}, it is clear that $\abs{S(k)}\leq w(\yb)$ holds for every $k\in\Sbb$ and $\xb\in\Sc_{\hgt}$.
At this time we can afford to be more generous to obtain a rougher estimate on the effect of the spin-flip part generator on $f_{\mathrm{CD}}$.
Taking into account \eqref{eq:lemmaS}, \eqref{eq:votergeneratoreffect} and assumption \eqref{assump:R21b} [\ref{A1}] we have the following bound
\[
|(\Grf f_{\mathrm{CD}})(\xb)| \leq \sum_{k\in\Sbb}(p_{k}(\xb)+q_{k+1}(\xb))\abs{S(k)}\leq 2w(\yb)^2,
\]
since $\abs{S(k)}\leq w(\yb)$ holds for any $k\in\Sbb$. Hence
\[
\abs{(\Grh f_{\mathrm{CD}})(\xb)} \leq
\alpha_1|(\Grf f_{\mathrm{CD}})(\xb)|+\alpha_2|(\Grex f_{\mathrm{CD}})(\xb)|
\leq2\alpha_1w(\yb)^2+\alpha_2B_{\abs{\yb}}\abs{\yb},
\]
where we also used \eqref{swappingpartcontr} with the rougher upper bounds provided by \eqref{assump:R22b}. Since $B_{\abs{\yb}}$ grows asymptotically at most linearly by \eqref{assump:rateBb} we have a $m>0$ and $n'\in\Nbb$ such that $\sup_{N\geq n'}\max_{1\leq n\leq N}B_n/N\leq \bar{D}$ and
\[
\abs{(\Grh f_{\mathrm{CD}})(\xb)} \leq 2\alpha_1w(\yb)^2 + \alpha_2m\abs{\yb} + \alpha_2\bar{D}\abs{\yb}^2
\leq \Gamma\;w(\yb)^2
\]
hold, where $\Gamma=2\alpha_1+\alpha_2m+\alpha_2\bar{D}$. From this, one can conclude that:
\begin{align}
|M^{(n)}(t)|\leq f_{\mathrm{CD}}\big(\Xb\big(t\wedge \tau_{W}^{(n)}\big)\big)+
\Gamma\int_{0}^t W(s)^2\,\dr s \leq \big(1  + \Gamma\,t\big)\,Q^{w_0}(t)^2,\label{ineq:unifint}
\end{align}
which is due to the fact that the number of inversions at time $s$ is bounded above by the squared maximal width $W(s)^2$ and that $W(s)\leq Q^{w_0}(s)\leq Q^{w_0}(t)$ holds a.s.\ in the coupling of the proof of Proposition \ref{propvariance} for every $0\leq s\leq t$. As a result of \eqref{ineq:unifint} and Proposition \ref{propvariance} the process $(M^{(n)}(t))_{t\geq0}$ is uniformly integrable in $n$, hence
\[
\Exp[f_{\mathrm{CD}}(\Xb_0)]+\int_0^t\Exp[(\Grh f_{\mathrm{CD}})(\Xb(s))]\,\dr s=\Exp[f_{\mathrm{CD}}(\Xb(t))] \geq \Exp[W(t)]-2,
\]
where the last inequality follows from the simple estimate $f_{\mathrm{CD}}(\xb)\geq w(\yb)-2$.
\end{proof}
The proof of Theorem \ref{tightness} will go by contradiction, the essence of which is formulated in Proposition \ref{veryuntight} below.
Before that we prove a lemma that we will need subsequently.
\begin{lemma}\label{proplemma}
Suppose that $\alpha_1>0$, \ref{A0}, \eqref{assump:R21a} from \ref{A1} and
\begin{equation}\label{ind1}
\int_0^T\Prob\big\{\abs{\Yb(t)}<N\,\big\}\,\dr t=\ordo(T)
\end{equation}
hold for some odd positive integer $N$ and $\Yb(0)\in\Sc$. Then we have the following:
\begin{align}
\int_0^T\Prob\big\{\exists s\in[0,S]& \text{ such that } \abs{\Yb(t+s)} < N\big\}\,\dr t=\ordo(T),\text{ and }\label{ind2}\\
\int_{0}^T\Prob\big\{\exists s\in[0,S]
\text{ such that }
&\abs{\Yb(t+s)}=N\text{ and }Y_{i_1}(t+s)=1=-Y_{i_2}(t+s)\nonumber\\
&\text{ for some $i_1,i_2\in\Zbb$ such that $|i_1-i_2|\leq L$}\big\}\,\dr t=\ordo(T)\label{ind3}
\end{align}
for every fixed $S>1$ and $L\in\Zbb^+$.
\end{lemma}
\begin{proof}[Proof of Lemma \ref{proplemma}.]
First we define the following stopping time:
\begin{equation}\label{tauidountight}
\tau_n(t)=\inf\{u \geq 0 : \abs{\Yb(t+u)}<n\},
\end{equation}
for each fixed $t>0$, where $n$ is an odd integer. Notice that under the assumptions $\alpha_1>0$, \ref{A0} and \eqref{assump:R21a}, we have:
\[
0<\inf\left\{\Prob\left(\abs{\Yb(t)}< N \mbox{ for every $t\in[0,1]$}\right):|\Yb(0)| < N\right\}=:p,
\]
where $\Yb(0)\in\Sc$. Furthermore defining $A_S(t)$ to be the event of having at most $N-2$ particles during a time interval of length one somewhere between time $t$ and $t+S+1$, it follows by the strong Markov property that:
\begin{align*}
&\Prob(A_S(t)\,|\,\exists s\in[0,S] \text{ such that } \abs{\Yb(t+s)}<N)\\
&\geq\Prob(\abs{\Yb(u)}<N \text{ $\forall u\in[t+\tau_N(t),t+\tau_N(t)+1]$}\,\vphantom{\frac{1}{1}}|\,\tau_N(t)\leq S)\geq p.
\end{align*}
Using the notation $B_{i,n}(t)$ for the event $\{|\Yb(t+i S/n-S/n^2)|\geq N\}$ ($i\leq n$), the estimates follow:
\begin{align}
p\int_{0}^T
&\Prob\left(\exists s\in[0,S]:\abs{\Yb(t+s)} < N\right)\,\dr t\leq\int_0^T\Prob(A_S(t))\,\dr t\nonumber\\
&\leq\int_0^T\Prob\bigg(\bigcup_{i=0}^{n}\left\{\abs{\Yb\left(t+i S/n\right)}<N\right\}\bigg)\,\dr t\nonumber\\
&\leq\sum_{i=1}^{n}\int_0^T\big[\Prob(\abs{\Yb\left(t+i S/n\right)}<N\,|\,B_{i,n}(t))
+\Prob(\abs{\Yb\left(t+i S/n\right)}<N)\big]\,\dr t\nonumber\\
&\leq nT(\alpha_1N+\alpha_2B(N))\frac{S}{n^2}+nT\,\ordo\bigg(\frac{1}{n^2}\bigg)+n\int_0^{T+S}\Prob(\abs{\Yb(t)}<N)\,\dr t
\label{limsup}
\end{align}
for every large $n\geq \left\lceil2S\right\rceil$.
At the last inequality we took advantage of that $(\Yb(t))_{t\geq0}$ is a continuous time Markov process in which a step or branch is executed with rate at most $\alpha_1N+\alpha_2B(N)$, if at most $N$ particles are present.
Divide the left hand-side of \eqref{ind2} by $T$ and use the previous estimate:
\begin{align}
\lim_{T\to+\infty}\frac{1}{T}\int_{0}^T\!\Prob(\tau_N(t)\leq S)\,\dr t
&\leq \frac{S(\alpha_1N+\alpha_2B(N))+1}{p}\frac{1}{n}\nonumber\\
&+\frac{n}{p}\lim_{T\to+\infty}\left(1+\frac{S}{T}\right)\frac{1}{T+S}\int_0^{T+S}\!\!\!\Prob(\abs{\Yb(t)}<N)\,\dr t,\label{limsuputan}
\end{align}
which tends to $0$ as $n\to+\infty$ by \eqref{ind1} ($N$ and $S>0$ are fixed). This shows \eqref{ind2}.

The next step is to prove \eqref{ind3} in a similar fashion. For sake of simplicity we introduce some further notation:
\begin{align*}
D(t):=\{\exists s\in[0,S]\text{ such that }
&\abs{\Yb(t+s)}=N\text{ and }\\
&\,Y_{i_1}(t+s)=1=-Y_{i_2}(t+s)\text{ for some }|i_1-i_2|\le L\}.
\end{align*}
As before by \ref{A0}, \eqref{assump:R21a} (see \ref{A1}) and $\alpha_1>0$, it is straightforward that
\begin{align*}
0<\inf\{&\Prob\{\abs{\Yb(t)}<N\text{ for every $t\in[1,2]$}\}:\\
&\,|\Yb(0)| = N,Y_{i_1}(0)=1=-Y_{i_2}(0)\text{ for some }|i_1-i_2|\le L\}=:q,
\end{align*}
which implies that $\Prob(A_{S+1}(t)\,|\,D(t))\geq q$. Next we proceed by \eqref{limsup}, and use that $2S>S+1$:
\begin{align*}
\frac{q}{T}\int_{0}^T\Prob(D(t))\,\dr t
&\leq\frac1T\int_0^T\Prob(A_{2S}(t))\,\dr t\\
&\leq (\alpha_1N+\alpha_2B(N))\frac{2S}{n}+\ordo\bigg(\frac{1}{n}\bigg)+\frac{n}{T}\int_0^{T+2S}\!\!\Prob(\abs{\Yb(t)}<N)\,\dr t.
\end{align*}
Now dividing by $q$ and making the same changes as we made in \eqref{limsuputan}, finally taking the limits first as $T$, then $n\to+\infty$ we obtain \eqref{ind3}. We notice that throughout the computations the constants $N$, $S>0$ and $L\in\Zbb^+$ were fixed.
\end{proof}
\begin{prop}\label{veryuntight}
Suppose that $\alpha_1>0$, \ref{A0}, \eqref{assump:R21a} from \ref{A1}, \eqref{assump:R22}, \eqref{assump:rateB} from \ref{A2}, and \ref{A3} hold. Then if $(\wtild{\Yb}(t))_{t\geq0}$ was not positive recurrent then we would have
\begin{equation}\label{cesaroconv}
\lim_{T\to+\infty}\frac{1}{T}\int_0^T\Prob\big\{\abs{\Yb(t)}<N\,\big\}\,\dr t=0
\end{equation}
for every odd integer $N$.
\end{prop}
\begin{proof}[Proof of Proposition \ref{veryuntight}.]
The basic idea of our argument has its roots in the proof of Proposition 13 in \cite[pp.\ 80--83]{SS08V}. We use induction and follow similar steps to see that \eqref{cesaroconv} holds. The heuristic idea behind the proof is as follows: if $(\wtild{\Yb}(t))_{t\geq0}$, the process $(\Yb(t))_{t\geq0}$ as seen from its leftmost particle position, was not positive recurrent, then it would be an unlikely event that on the long run we have only a few particles. They even must be far from each other, otherwise nearby particles would consume each other so we would end up with fewer particles with high probability. Now, few particles situated far from each other implies that they must keep growing due to the fact that distant particles do not feel each other and a lone particle would result a cloud of at least three particles with high probability. It means that eventually the probability of having only a few particles on the long run would remain small.

By translation invariance, which is formulated in \ref{A0}, the claim \eqref{cesaroconv} trivially follows if $N\in\{1,\,3\}$. Hereinafter for an $A\in\mathbb{R}^+$ we denote by $\sigma_A$ a uniformly distributed random variable over $[0,A]$ which is independent of the process $(\Yb(t))_{t\geq0}$ and of other $\sigma_{A'}$'s. We proceed inductively: assuming that \eqref{cesaroconv} holds for an odd integer $N\geq3$, we prove that
\begin{equation}
\Exp\bigg[\frac{1}{T}\int_0^T\ind\big\{\abs{\Yb(t)}< N+2\;\big\}\,\dr t\bigg]=
\Prob\big\{\abs{\Yb(\sigma_T)}< N+2\;\big\}\longrightarrow 0\label{finalstatement}
\end{equation}
as $T\to+\infty$.

By the induction hypothesis we also have \eqref{ind2} and \eqref{ind3} in the bag from Lemma \ref{proplemma}. We will prove \eqref{finalstatement} through series of statements in the following strategy: from the former ones we will conclude for \eqref{ind4} below (Step I), this will imply \eqref{ind5} (Step II), which in turn provides us \eqref{finalstatement} (Step III).

\noindent {\bf Step I (\eqref{ind2},\eqref{ind3} $\Rightarrow$ \eqref{ind4}).} We show that
\begin{equation}
\lim_{S\to+\infty}\limsup_{T\to+\infty}\frac{1}{T}\int_{0}^T
\Prob\left(\tau_{N+2}(t)\leq S,\,\abs{\Yb(t+\tau_{N+2}(t)+\sigma_S)} < N+2\right)\,\dr t=0,\label{ind4}
\end{equation}
where the stopping time $\tau_N$ was defined in \eqref{tauidountight}.\\
\noindent {\bf Step I/a}. As the first ingredient to accomplish Step I, we prove that for every $\varepsilon>0$ there exists an $S_0$ such that
\begin{equation}\label{YbpicistepIa}
\Prob(\abs{\Yb(\sigma_S)} \leq N)\leq \varepsilon
\end{equation}
holds for every $S\geq S_0$, where $\Yb(0)=\Yb_0$,
\[
\Yb_0=\sum_{j=0}^{N-1}(-1)^{j + (1-\ch(\Yb_0))/2}\,\delta_{i_{\mathrm{right}}-j\uip{\exp(L)}},
\]
$i_{\mathrm{right}}\in\Zbb$, and $L=L(S)\in\Zbb^+$ may depend on $S$.

Let $\hat{\Yb}_0$ be the configuration which contains only the rightmost (leftmost) particle of $\Yb_0$, if $\ch(\Yb_0)=+1$ ($\ch(\Yb_0)=-1$). Let $(\hat{\Yb}(t))_{t\geq0}$ be the process of initial configuration $\hat{\Yb}(0)=\hat{\Yb}_0$ and of dynamics governed by the rates $(r_i,\ell_i,b_i)_{i\in\Zbb}$. Hence, we know that for every $0<\varepsilon<1$ one can choose an $S_0=S_0(\varepsilon)$ such that
\begin{equation}\label{hatYbpici}
\Prob\big(\big|\hat{\Yb}(\sigma_S)\big|\geq 3\big)>1-\frac{\varepsilon}{4}
\end{equation}
holds for every $S\geq S_0$. In what follows by assumption \ref{A3} we will approximate the rightmost (leftmost) branch of particles of $(\Yb(t))_{t\geq0}$ by the process $(\hat{\Yb}(t))_{t\geq0}$ if $L$ is chosen sufficiently large in a finite time interval $[0,S]$, where $\ch(\Yb_0)=+1$ ($\ch(\Yb_0)=-1$). For an illustration, see Figure \ref{fig:dbarw-proof-prop}. We proceed with a series of definitions.

For $j\in\{1,2,\ldots,N\}$ let $\Yb_{j}=(\Yb_{j})_{t\geq0}$ be the system of those particles which are descendants of the $j^{\mathrm{th}}$ particle of $\Yb_0$ during the process $(\Yb(t))_{t\geq0}$. Furthermore let $\nu_L$ be the first time when the rightmost (leftmost) particle of $\Yb_j$ for some $j$ gets to distance $L$ from the leftmost (rightmost) particle of $\Yb_{j+1}$ ($\Yb_{j-1}$), where $1\leq j<N$ ($1<j\leq N$). By the definition of $\Yb_0$ it is clear that $\nu_L$ is almost surely positive for every $L\in\Zbb^+$. Using a similar argument we established in the proof of Proposition \ref{propvariance}, one can choose an $L_0\in\Zbb^+$ such that
\begin{equation}\label{gapisbig}
\Prob(\nu_{L_0} \leq S)\leq \frac{\varepsilon}{4}
\end{equation}
holds for every $L>L_0$, since $\exp(L)-L\to+\infty$ as $L\to+\infty$. We will make the choice for $L$ later.
\begin{figure}[!ht]
\centering
\includegraphics[scale=0.3]{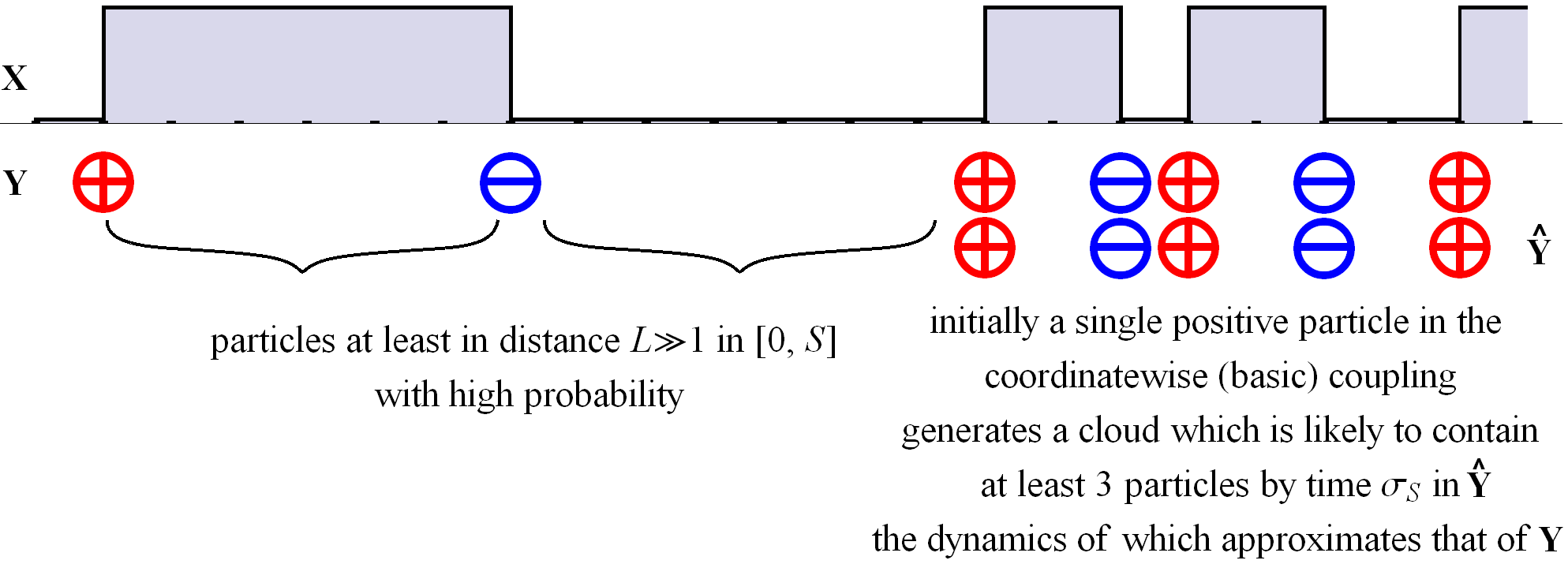}%
\caption[The essence of proof of Proposition \ref{veryuntight}.]{The rough idea behind the proof of Proposition \ref{veryuntight} in the case of $\ch(\Yb_0)=+1$.}\label{fig:dbarw-proof-prop}%
\end{figure}

Our next step is to set up the basic coupling between the processes $(\Yb(t))_{t\geq0}$ and $(\hat{\Yb}(t))_{t\geq0}$. That is we place the two systems one above the other and we let a (possible) particle at position $i\in\Zbb$ to step and branch with rates
\begin{equation}
\big(r_{i}^{\circ}(\Yb)-r_i^{\circ}({\hat{\Yb}})\big)_{+},\;\big(\ell_{i}^{\circ}(\Yb)-\ell_i^{\circ}({\hat{\Yb}})\big)_{+}\text{ and }
\big(b_{i}^{\circ}(\Yb)-b_i^{\circ}({\hat{\Yb}})\big)_{+}\label{eq:compcl1}
\end{equation}
in $\Yb$, while
\begin{equation}
\big(r_i^{\circ}({\hat{\Yb}})-r_{i}^{\circ}(\Yb)\big)_{+},\;\big(\ell_i^{\circ}({\hat{\Yb}})-\ell_{i}^{\circ}(\Yb)\big)_{+}\text{ and }
\big(b_i^{\circ}({\hat{\Yb}})-b_{i}^{\circ}(\Yb)\big)_{+}\label{eq:compcl2}
\end{equation}
in $\hat{\Yb}$, respectively, while these moves occur simultaneously in the two processes with rates
\[
\min\big(r_{i}^{\circ}(\Yb),\ r_i^{\circ}({\hat{\Yb}})\big),\;\min\big(\ell_{i}^{\circ}(\Yb),\ \ell_i^{\circ}({\hat{\Yb}})\big)\text{ and }\min\big(b_{i}^{\circ}(\Yb),\ b_i^{\circ}({\hat{\Yb}})\big),
\]
where $\circ\in\{\opl,\omi\}$. The possible steps indicated by \eqref{eq:compcl1} and \eqref{eq:compcl2} will be called ``compensating'' if and only if at least one of the rates $r_i^{\circ}(\hat{\Yb}),\ell_i^{\circ}(\hat{\Yb})$ or $b_i^{\circ}(\hat{\Yb})$ is positive. The very first time of a compensating step is denoted by $\xi$, while $A_k$ represents the event that the first compensating step that is part of the evolution of the joint system $(\Yb,\hat{\Yb})$ happens to be its $k^{\mathrm{th}}$ step in total.

As the final preparatory step recall that $(N(s))_{0\leq s\leq S}$ counts the number of steps taken by the process $(\Yb(t))_{t\geq0}$ in the time interval $[0,S]$ (see its definition in Proposition \ref{propwelldefined}). In accordance with this the number of steps taken by $(\hat{\Yb}(s))_{0\leq s\leq S}$ is $(\hat{N}(s))_{0\leq s\leq S}$.

Now, one can choose an $M>1$ large enough for
\begin{equation}\label{stepsarenotbig}
\Prob(N(S)+\hat{N}(S) \geq M) \leq \Prob(N(S)\geq M/2) + \Prob(\hat{N}(S)\geq M/2)\leq\frac{\varepsilon}{4}
\end{equation}
to hold. The choice of $M$ can be made independently of that of $L$, since we assumed in \eqref{assump:R22b} that independently of the situation of $k$ particles dropped on the integer lattice the total rate of branching is at most $B(k)$. In fact the estimates coming from Proposition \ref{propwelldefined} do not use the underlying structure of particles. Hence the value of $M$ can be fixed from now. Let $\omega_M$ be the first time when the process $(N(s)+\hat{N}(s))_{0\leq s\leq S}$ exceeds the value $M$. It follows that for every $0\leq s\leq \omega_M$:
\[
\max\big(\abs{\Yb(s)},\big|\hat{\Yb}(s)\big|\,\big)\leq N+1+2M,
\]
since the particle number can grow at most by two due to branching at each step. So until the $M^{\mathrm{th}}$ transition the total number of particles can not exceed the value $N+1+2M$ providing initially at most $N+1$ of them was present.

\noindent Now, using the previous notations we arrive to the following estimates:
\begin{align}
\Prob\big(\xi \leq \min(S,\nu_L,\omega_M)\big)&=\sum_{k=1}^M\Prob\big(A_k,\xi \leq \min(S,\nu_L,\omega_M)\big)\nonumber\\
&\leq M (N+1+2M)\,\delta(L) \leq \frac{\varepsilon}{4},\label{couplingestimate}
\end{align}
where
\[
\frac{\alpha_1+\alpha_2}{\alpha_1\ul{s}}\cdot\max_{1\leq m \leq N+3M} H(m, L) =: \delta(L) \leq \frac{\varepsilon}{4 M (N+1+2M)},
\]
providing that $\alpha_1>0$, which latter inequality holds if $L$ is chosen big enough. In plain words the probability that until $\min(S,\nu_L,\omega_M)$ a compensating clock of a particle rang is at most $\delta(L)$, which can be made arbitrarily small if $L>L_0$ by assumption \ref{A3} (see \eqref{assump:3decay}). Now, putting the above pieces together we arrive to:
\begin{align*}
\Prob(\abs{\Yb(\sigma_S)}\leq N) \leq
&\Prob(\abs{\Yb(\sigma_S)}\leq N,\nu_L \leq S)+\Prob(\abs{\Yb(\sigma_S)}\leq N,\nu_L > S)\\
\leq
&\Prob(\nu_L\leq S)+\Prob(\abs{\Yb(\sigma_S)}\leq N,\nu_L>S,\omega_M \leq S)\\
&+\Prob(\abs{\Yb(\sigma_S)}\leq N,\nu_L>S,\omega_M>S)\\
\leq
&\Prob(\nu_L\leq S)+\Prob(\omega_M \leq S)\\
&+\Prob(\abs{\Yb(\sigma_S)}\leq N,\nu_L>S,\omega_M>S,\xi\leq S)\\
&+\Prob(\abs{\Yb(\sigma_S)}\leq N,\nu_L>S,\omega_M>S,\xi > S)\\
\leq
&\Prob(\nu_L\leq S)+\Prob(\omega_M \leq S)\\
&+\Prob(\xi \leq \min(S,\nu_L,\omega_M))+\Prob\big(|\hat{\Yb}(\sigma_S)| < 3\big)\\
&+\Prob(\abs{\Yb(\sigma_S)}\leq N, |\hat{\Yb}(\sigma_S)| \geq 3, \nu_L > S, \xi > S)\\
\leq&\frac{\varepsilon}{4}+\frac{\varepsilon}{4}+\frac{\varepsilon}{4}+\frac{\varepsilon}{4}=\varepsilon,
\end{align*}
since the last probability on the right-hand side is zero and the preceding ones are bounded above by $\varepsilon/4$ using \eqref{gapisbig}, \eqref{stepsarenotbig}, \eqref{couplingestimate} and \eqref{hatYbpici} in that order. So we can conclude that \eqref{YbpicistepIa} holds.

\noindent {\bf Step I/b}. Recall that so far we proved that for every $0<\varepsilon<1$ one can choose an $S_0$ such that for every $S\geq S_0$:
$\Prob(\abs{\Yb(\sigma_S)} \leq N)\leq\varepsilon$ holds for each $\Yb(0)$, where $\abs{\Yb(0)}=N$ and the distance between any of its two particles is at least $L'=\uip{\exp(L)}$, where $L'$ may depend on $S$. Now, notice that
\begin{align}
\{&\tau_{N+2}(\sigma_T)\leq S\}\subset\nonumber\\
&\;\{\tau_N(\sigma_T)\leq S\}
\cup\{\exists\,s\in[0,S]:\abs{\Yb(\sigma_T+s)}=N\text{ and there are two $L'$-close particles}\}\nonumber\\
&\;\cup\{\Yb(\sigma_T+\tau_{N+2}(\sigma_T))\text{ contains exactly $N$ particles and between any two of them,}\nonumber\\
&\hspace{3.6cm}\text{ the distance is at least $L'$}\}.\label{propstepIb}
\end{align}
Thus by what we have just proved in Step I/a, the probability that
\[
\abs{\Yb(\sigma_T+\tau_{N+2}(\sigma_T)+\sigma_S)}\leq N
\]
occurs, given that the last event of the right-hand side of \eqref{propstepIb} holds, is uniformly bounded (in $T$) above by $\varepsilon$ if $S\geq S_0$ and $L'$ is chosen properly.
Take into account \eqref{ind2} and \eqref{ind3}, which tell us that the probabilities,
$\tau_{N}(\sigma_T)$ is at most $S$ while $\Yb(\sigma_T+\tau_{N+2}(\sigma_T))$ consists of $N$ particles such that there exist two particles which are in distance at most $L'$, are small. These and \eqref{propstepIb} together imply that
\[
\limsup_{T\to+\infty}\Prob(\tau_{N+2}(\sigma_T)\leq S,\,\abs{\Yb(\sigma_T+\tau_{N+2}(\sigma_T)+\sigma_S)}\leq N) \leq \varepsilon,
\]
thus we are ready with the proof of \eqref{ind4}.

\noindent {\bf Step II (\eqref{ind4} $\Rightarrow$ \eqref{ind5}).} In the penultimate step from \eqref{ind4} we conclude that:
\begin{equation}
\lim_{S\to+\infty}\limsup_{T\to+\infty}\frac{1}{T}\int_{0}^T\frac{1}{S}\int_{0}^S\Prob\left(
\abs{\Yb(t+s)}\leq N\right)\,\dr s\,\dr t=0.\label{ind5}
\end{equation}
A straightforward computation yields:
\begin{align*}
\int_{0}^S\Prob(\abs{\Yb(t+s)}\leq N)\,\dr s
=&\Exp\left[\ind\{\tau_{N+2}(t)\leq S\}\int_{\tau_{N+2}(t)}^S\ind\{\abs{\Yb(t+s)}\leq N\}\,\dr s\right]\\
\leq&\Exp\left[\ind\{\tau_{N+2}(t)\leq S\}\int_{\tau_{N+2}(t)}^{\tau_{N+2}(t)+S}\ind\{\abs{\Yb(t+s)}\leq N\}\,\dr s\right]\\
=&\int_{0}^{S}\Prob\big(\tau_{N+2}(t)\leq S,\,\abs{\Yb(t+\tau_{N+2}(t)+s)}\leq N\big)\,\dr s.
\end{align*}
After dividing by $T$ and $S$ and taking the appropriate limits we obtain \eqref{ind5}.

\noindent {\bf Step III (\eqref{ind5} $\Rightarrow$ \eqref{finalstatement}).}
For every $T>0$ one can choose an $S(T)$ due to \eqref{ind5} such that $T>S(T)$,
\[
\lim_{T\to+\infty}\frac{S(T)}{T}=0,\mbox{ and }\lim_{T\to+\infty}\Prob\big(\abs{\Yb(\sigma_T+\sigma_{S(T)})} < N+2\big)=0\mbox{ holds}.
\]
Let
\[
f_{T,S}(x):=\frac{T+S-x+(x-S)\ind\{x\leq S\}+(x-T)\ind\{x\leq T\}}{TS}\qquad(0\le x\le T+S).
\]
It is easy to see that $f_{T,S}$ is the probability density function of the random variable $\sigma_T+\sigma_S$. It follows that
\begin{align*}
&\abs{\Prob\big(\abs{\Yb(\sigma_{T}+\sigma_{S(T)})} < N+2\big)-\Prob\big(\abs{\Yb(\sigma_{T})} < N+2\big)}\\
&\leq\int_{0}^{T+S(T)}\Prob(\abs{\Yb(t)}\leq N)\abs{f_{T,S(T)}(t)-\frac1T\ind\{t\leq T\}}\,\dr t\\
&\leq\int_{0}^{T+S(T)}\bigg[\frac{S(T)-t}{T S(T)}\ind\{t\leq S(T)\}+\frac{T+S(T)-t}{T S(T)}\ind\{t>T\}\bigg]\,\dr t=\frac{S(T)}{T},
\end{align*}
which tends to $0$ as $T\to+\infty$, that is we finished the proof of the proposition.
\end{proof}

Now we are ready to prove the main theorem of Subsection \ref{sec:tightness}.
\begin{proof}[Proof of Theorem \ref{tightness}.]
Lemma \ref{lemma} and \ref{lemmagenerator} imply the following inequality:
\begin{equation}\label{eq:firstintheorem}
0 \leq \Exp[W(t)] \leq \big(2+f_{\mathrm{CD}}(\Xb_0) + Ct\big) - c \int_0^t\Exp[\abs{\Yb(s)}]\,\dr s,
\end{equation}
which holds for every $t\geq0$, where we used the same constants $c,C\in\Rbb^+$ as were stated in Lemma \ref{lemma}. Rearranging \eqref{eq:firstintheorem} we obtain
\begin{equation}\label{eq:thmmain}
\int_0^T\Exp[\abs{\Yb(s)}]\,\dr s \leq C_0 T,
\end{equation}
which holds for every $T\geq1$, where $C_0=\frac{1}{c}[2+f_{\mathrm{CD}}(\Xb_0)+C]>0$. Applying Markov's inequality on \eqref{eq:thmmain} we obtain that
\begin{equation}\label{eq:secondintheorem}
1-\frac{C_0}{N}\leq\frac{1}{T}\int_{0}^T\Prob(\abs{\Yb(t)}<N)\,\dr t
\end{equation}
holds for every odd integer $N>0$ and $T\geq1$. At this point we choose $N$ to be strictly greater than $C_0$ and we fix it. Finally we argue in the following manner to complete the proof: if $\wtild{\Yb}$ ($\wtild{\Xb}$), the process $\Yb$ ($\Xb$) as seen from its leftmost particle position, was not positive recurrent on the appropriate configuration space, then by Proposition \ref{veryuntight} the right-hand side of \eqref{eq:secondintheorem} would have a vanishing limit as $T\to+\infty$, which would lead to an obvious contradiction, since the left-hand side of \eqref{eq:secondintheorem} remains positive by choice.
\end{proof}
\begin{proof}[Proof of Corollary \ref{corollary:tightness}.]
The first assertion is a straightforward consequence of Theorem \ref{tightness}. For the second claim let
\[
V_N(\xb):=\min\big(f_{\mathrm{CD}}(\xb),N\big),
\]
where $N\in\Nbb$ and $\xb\in\Sc_{\hgt}$.
We denote $\wtild{\Gr}_{\hgt}$ by the infinitesimal generator of the process $(\wtild{\Xb}(t))_{t\geq0}$. Notice that
\begin{equation}\label{eq:ketGugyanaz}
\big(\wtild{\Gr}_{\hgt} f_{\mathrm{CD}}\big)(\wtild{\xb})=\big(\Grh f_{\mathrm{CD}}\big)(\xb)
\end{equation}
holds for every $\xb\in\Sc_{\hgt}$ and $\wtild{\xb}\in\wtild{\Sc}_{\hgt}$ by translation invariance, where configuration $\wtild{\xb}$ is $\xb$, as seen from its leftmost $(1+\ch(\yb))/2$ value ($\xb$ is connected to $\yb$ through \eqref{spatderiv}). From Lemma \ref{lemma} we know that $0\leq C-\big(\wtild{\Gr}_{\hgt}V_N\big)(\wtild{\xb})$, where $\wtild{\xb}\in\wtild{\Sc}_{\hgt}$. So taking into account Lemma \ref{lemma}, \eqref{eq:ketGugyanaz} and Fatou's lemma we obtain the following chain of inequalities:
\begin{align*}
0 < \sum_{\wtild{\yb}\in\wtild{\Sc}}\abs{\wtild{\yb}}\,\pi(\wtild{\yb})
&\leq
\sum_{\wtild{\xb}\in\wtild{\Sc}_{\hgt}}\!\bigg(\frac{C}{c}-\frac{1}{c}\big(\Grh f_{\mathrm{CD}}\big)(\wtild{\xb})\bigg)\,\pi(\wtild{\xb})\\
&=\sum_{\wtild{\xb}\in\wtild{\Sc}_{\hgt}}\!\lim_{N\to+\infty}\bigg[\frac{C}{c}-\frac{1}{c}\big(\wtild{\Gr}_{\hgt}V_{N}\big)(\wtild{\xb})\bigg]
\,\pi(\wtild{\xb})\\
&\leq \liminf_{N\to+\infty}
\bigg[\frac{C}{c}-\frac{1}{c}
\sum_{\wtild{\xb}\in\wtild{\Sc}_{\hgt}}\!\!\big(\wtild{\Gr}_{\hgt}V_{N}\big)(\wtild{\xb})\;\pi(\wtild{\xb})\bigg]=\frac{C}{c}<+\infty,
\end{align*}
where at the last equality we took advantage of that $\pi_{\bll}$ is invariant under $\wtild{\Gr}_{\hgt}$ and $V_{N}$ is a bounded function. The connection between configurations $\wtild{\yb}\in\wtild{\Sc}$ and $\wtild{\xb}\in\wtild{\Sc}_{\hgt}$ above is established by the relation \eqref{spatderiv}.

The only thing left unproven is the asymptotics on the width process, which indeed follows directly from formula \eqref{eq:firstintheorem}, that is
\[
\limsup_{t\to+\infty}\frac{\Exp[W(t)]}{t} \leq C<+\infty.
\]
\end{proof}

\bigskip

\begin{proof}[Proof of Theorem \ref{non-nearesttightness}.]
We are going to follow the lines of the proof of Theorem \ref{tightness}, through the proof of Propositions \ref{propwelldefined} and \ref{propvariance}, then the lemmas and Proposition \ref{veryuntight}. Clearly Lemma \ref{lemma} lies at the heart of the method, the rest can be carried out easily. First of all we plug $f_{\mathrm{CD}}$ into $\Grh^{\exl}$ to obtain
\begin{align}
(\Grh^{(\mathrm{LR})}f_{\mathrm{CD}})(\xb)
\leq \big(C-c\abs{\yb}\big)+\sum_{l=2}^{+\infty}l^2\wtild{B}(l),\label{lastsum}
\end{align}
where we used the fact that
\[
(\Grh^{\exl}f_{\mathrm{CD}})(\xb)\leq l^2\max_{i\in\Zbb}(b_{i,l}(\xb))\leq l^2 \wtild{B}(l)
\]
by \ref{A5} and the previously computed bound from Lemma \ref{lemma}. Using again \ref{A5}, the last sum of \eqref{lastsum} is bounded above by an absolute constant, that is
\[
(\Grh^{(\mathrm{LR})}f_{\mathrm{CD}})(\xb)\leq \bar{C}-c\abs{\yb}
\]
holds for some $c,\bar{C}\in\Rbb^+$. The next step is to prove that the width process of the parity conserving long range branching and annihilating random walk does not grow too fast. Let $(H_n)_{n\in\Nbb}$ be the following process:
\begin{equation*}
\Prob\bigg(H_{n+1}=
\left\{
  \begin{array}{l}
    \,2H_n\\
    \,H_n+1
  \end{array}
\right.
\;\bigg|\vphantom{\frac{1}{1}}\bigg.\,H_n\bigg)=
\left\{
  \begin{array}{l}
    \,\frac{D}{\alpha_1+D}; \\
    \,\frac{\alpha_1}{\alpha_1+D},
  \end{array}
\right.
\end{equation*}
where $H_0=w_0>1$, and $D=\sum_{l=2}^{+\infty}\wtild{B}(l)<+\infty$ by \ref{A5}. Now, take a sequence of independent exponentially distributed random variables $(\wtild{V}_n)_{n\in\Zbb^+}$ such that for an $n\in\Zbb^+$
\[
\Exp\big(\wtild{V}_{n}\big)=\frac{1}{2(\alpha_1+\alpha_2\bar{B}+D)},
\]
and set up $\wtild{S}_n$ to be $\wtild{V}_1+\cdots+\wtild{V}_n$, where $\wtild{S}_0=0$. Then for an $a>0$ let
\begin{equation*}
\tau_a=\min\{n\geq 0\,:\,\wtild{S}_n\geq a\}\mbox{ and }
\wtild{Q}^{w_0}(t):=H_{\tau_t}\quad(t\geq 0),
\end{equation*}
where $w_0=w(\Yb^{(\mathrm{LR})}(0))$ is the width of $\Yb^{(\mathrm{LR})}(0)\in\Sc$. Since the maximal rate with which the width grows by one is $2\alpha_1+2\alpha_2\bar{B}$, while it grows by more than one at most with rate $2\sum_{l=2}^{w_0+n}\wtild{B}_l\leq2 D$, we can set up a coupling such that $\wtild{Q}^{w_0}(t)$ stochastically dominates the width process of $\left(\Yb^{(\mathrm{LR})}(t)\right)_{t\geq0}$.

Turning to the proof of a claim similar to Lemma \ref{lemmagenerator} (see \eqref{lemmamarti}), by means of \eqref{ineq:unifint}, it is enough to prove that $\wtild{Q}^{w_0}(t)^2$ is of finite expectation. Of course even more is true, since for a given $r\in\Zbb^+$, the following estimate holds:
\begin{align*}
\Exp\big[\wtild{Q}^{w_0}(t)^r\big]
&\leq (w_0+1)^r\Exp\big[2^{r\tau_t}\big] \leq (w_0+1)^r\sum_{n\geq2}\Prob\big(\tau_t \geq \lip{\log_2(n)}/r\big)\\
&\leq (w_0+1)^r\sum_{n\geq2}\Prob\big(\wtild{V}_1+\cdots+\wtild{V}_{\lip{\log_2(n)}/r}<t\big)\\
&\leq (w_0+1)^r\exp(\lambda t)\sum_{n\geq 2}\Exp\big[\exp(-\lambda \wtild{V}_1)\big]^{\lip{\log_2(n)}/r}\\
&\leq (w_0+1)^r\exp(\lambda t)\sum_{n\geq 2} \bigg(\frac{2+2\bar{B}+2D}{2+2\bar{B}+2D+\lambda}\bigg)^{\lip{\log_2(n)}/r}\\
&\leq (w_0+1)^r\exp(\lambda t)\sum_{n\geq 1} \bigg(\frac{2^{r+1}(1+\bar{B}+D)}{2+2\bar{B}+2D+\lambda}\bigg)^{n/r},
\end{align*}
where $t\geq0$ and under way we used Markov's inequality with a fixed $\lambda>0$. One can choose $\lambda$ to be greater than $2^{2r}$ resulting that the sum in the last display be finite.

We leave Lemma \ref{proplemma} and Proposition \ref{veryuntight} to the reader as these statements carry over and can be proved along the same lines.
\end{proof}

\section*{Acknowledgement}

The authors thank valuable discussions on the topics of general interacting particle systems made with B\'alint T\'oth and J\'anos Kert\'esz, and on swapping voter models with Jan Swart. Support from the grants T\'AMOP-4.2.2.B-10/1--2010-0009, OTKA K100473 and K109684, and the Bolyai Scholarship of the Hungarian Academy of Sciences is gratefully acknowledged.


\addcontentsline{toc}{section}{References}
\begin{thebibliography}{99}
%
\bibitem[BMV07]{BMV07} S. Belhaouari, T. Mountford and G. Valle. \textit{Tightness for the interfaces of one-dimensional voter models}. Proc. London Math. Soc. {\bf 94}(3), pp. 421--442, 2007.
%
\bibitem[BFMP01]{BFMP01} V. Belitsky, P. A. Ferrari, M. V. Menshikov and S. Y. Popov. \textit{A mixture of the exclusion process and the voter model}. Bernoulli {\bf 7}(1), pp. 119--144, 2001.
%
\bibitem[BK11]{BK11} J. Blath, N. Kurt. \textit{On survival and extinction of caring double-branching annihilating random walk}. Elect. Comm. in Probab. {\bf 16}(26), pp. 271--282, 2011.
%
\bibitem[BCDFLS86]{BCDFL86} M. Bramson, P. Calderoni, A. De Masi, P. Ferrari, J. Lebowitz and R. H. Schonmann. \textit{Microscopic selection principle for a diffusion-reaction equation}. J. Stat. Phys. {\bf 45}, pp. 905--920, 1986.
%
\bibitem[BG85]{BG85} M. Bramson, G. Gray. \textit{The survival of branching annihilating random walk}. Probab. Theory Rel. {\bf 68}, pp. 447--460, 1985.
%
\bibitem[BWD91]{BWD91} M. Bramson, D. Wan-ding and R. Durrett. \textit{Annihilating branching processes}. Stoch. Proc. Appl. {\bf 37}(1), pp. 1--17, 1991.
%
\bibitem[CT96]{CT96} J. Cardy, U. T\"auber. \textit{Theory of branching and annihilating random walk}. Phys. Rev. Lett. {\bf 77}, pp. 4780--4783, 1996.
%
\bibitem[CT98]{CT98} J. Cardy, U. T\"auber. \textit{Field theory of branching and annihilating random walk}. J. Stat. Phys. {\bf 90}, pp. 1--56, 1998.
%
\bibitem[CD95]{CD95} J. T. Cox, R. Durrett. \textit{Hybrid zones and voter model interfaces}. Bernoulli {\bf 1}(4), pp. 343--370, 1995.
%
\bibitem[DFL86]{DFL86} A. De Masi, P. A. Ferrari and J. Lebowitz. \textit{Reaction-diffusion equations for interacting particle systems}. J. Stat. Phys. {\bf 44}, pp. 589--644, 1986.
%
\bibitem[G79]{Gri79} D. Griffeath. Additive and Cancellative Interacting Particle Systems. Lecture Notes in Mathematics Vol. 724. Springer-Verlag, 1979.
%
\bibitem[KL99]{KL99} C. Kipnis, C. Landim. Scaling Limits of Interacting Particle Systems. Springer, 1999.
%
\bibitem[L85]{Lig85} T. M. Liggett. Interacting Particle Systems. Springer, 1985.
%
\bibitem[L99]{Lig99} T. M. Liggett. Stochastic Interacting Systems: Contact, Voter and Exclusion Processes. Springer, 1999.
%
\bibitem[O08]{O08} G. \'Odor. Universality in Nonequilibrium Lattice Systems: Theoretical Foundations, World Scientific, 2008.
%
\bibitem[SS08T]{SS08T} A. Sturm, J. M. Swart. \textit{Tightness of voter model interfaces}. Elect. Comm. in Probab. {\bf 13}(16), pp. 165--174, 2008.
%
\bibitem[SS08V]{SS08V} A. Sturm, J. M. Swart. \textit{Voter models with heterozygosity selection}. Ann. Appl. Probab. {\bf 18}(1), pp. 59--99, 2008.
%
\bibitem[S00D]{S00D} A. W. Sudbury. \textit{Dual families of interacting particle systems on graphs}, J. Theor. Probab., {\bf 13}(3), pp. 695-716, 2000.
%
\bibitem[S90]{Sud90} A. W. Sudbury. \textit{The branching annihilating process: An interacting particle system}. Ann. Probab. {\bf 18}, pp. 581--601, 1990.
%
\bibitem[S00S]{S00S} A. W. Sudbury. \textit{The survival of nonattractive interacting particle systems on $\Zbb$}, Ann. Probab., {\bf 28}(3), pp. 1149--1161, 2000.
\end{thebibliography}
\end{document}